\documentclass{amsart}

\usepackage{latexsym}
\usepackage{amsmath,enumerate}
\usepackage[psamsfonts]{amssymb}
\usepackage{amsthm}
\usepackage[mathscr]{eucal}

\usepackage{graphics}
\usepackage{tikz}

\usetikzlibrary{arrows}

\newtheorem{theorem}{Theorem}[section]
\newtheorem{lemma}[theorem]{Lemma}
\newtheorem{proposition}[theorem]{Proposition}

\theoremstyle{definition}
\newtheorem{definition}[theorem]{Definition}

\theoremstyle{remark}
\newtheorem{remark}[theorem]{Remark}

\numberwithin{equation}{section}

\def\1{\raisebox{2pt}{\rm{$\chi$}}}
\def\a{{\bf a}}
\def\b{{\bf b}}

\def\z{{\bf z}}

\def\R{\mathbb{R}}
\def\NN{\mathbb{N}}
\def\div{\mbox{div}\, }

\renewcommand{\L}{\mathcal{L}}
\renewcommand{\H}{{\mathcal H}}
\newcommand{\res}               {\!\!\mathop{\hbox{
                                \vrule height 7pt width .5pt depth 0pt
                                \vrule height .5pt width 6pt depth 0pt}}
                                \nolimits}

\begin{document}

\title[Pattern formation in a flux limited equation of porous media type]{Pattern formation in a flux limited reaction-diffusion equation of porous media type}

\author{J. Calvo}
\address{Departamento de Tecnolog\' ia, Universitat Pompeu-Fabra. Barcelona, SPAIN}
\email{juan.calvo@upf.edu}
\thanks{J. Calvo, O. S\'anchez and J. Soler were supported in part by MICINN (Spain), project
 MTM2011-23384, and Junta de Andaluc\'{\i}a Project  P08-FQM-4267. J. Calvo is also partially supported by a Juan de la Cierva grant of the Spanish MEC.}

\author{J. Campos}
\address{Departamento de Matem\'atica Aplicada,
Facultad de Ciencias, Universidad de Granada. 18071 Granada, SPAIN}
\email{campos@ugr.es}
\thanks{J. Campos was supported in part by MICINN (Spain), project MTM2011-23652.}

\author{V. Caselles}
\address{Departamento de Tecnolog\' ia, Universitat Pompeu-Fabra. Barcelona, SPAIN}
\email{vicent.caselles@upf.edu}
\thanks{V. Caselles  was supported in part by MICINN (Spain), project MTM2009-08171, and
 also acknowledges the partial
support by GRC reference 2009 SGR 773, and by ''ICREA Acad\`emia'' prize for excellence in research funded both
by the Generalitat de Catalunya.}

\author{O. S\'anchez}
\address{Departamento de Matem\'atica Aplicada,
Facultad de Ciencias, Universidad de Granada. 18071 Granada, SPAIN}
\email{ossanche@ugr.es}

\author{J. Soler}
\address{Departamento de Matem\'atica Aplicada,
Facultad de Ciencias, Universidad de Granada. 18071 Granada, SPAIN}
\email{jsoler@ugr.es}



\begin{abstract}
A nonlinear PDE featuring
flux limitation effects together with those
of the porous media equation (nonlinear Fokker--Planck) is presented in
this paper. We analyze the balance of such diverse effects through the
study of the existence and qualitative behavior of some admissible patterns, namely traveling wave solutions, to this singular reaction-diffusion
equation. We show the existence and qualitative behavior of different types of traveling waves:
classical profiles  for wave speeds high enough,
and discontinuous waves that are reminiscent of hyperbolic shock waves
when the wave speed lowers below a certain threshold.
Some of these solutions are of particular relevance
as they provide models by which the whole solution (and not just the bulk
of it, as it is the case with classical traveling waves) spreads through the medium
with
finite speed.
\end{abstract}

\maketitle

\section{Introduction, entropy solutions, main results}

Reaction--diffusion  equations assume that the behavior of the various populations described is ruled essentially by two processes: local reactions, in which the populations
interact between themselves, and diffusion, which makes the populations spread
out in the physical space. The concept of population is understood here
quite loosely, and several important examples can be found in developmental  biology, ecology,
geology, combustion theory, physics or computer sciences. Particles, free surface water waves, flames, cells, {bacteria} or morphogen concentrations in chemical
processes may qualify as such, see for instance \cite{AW2,Co,GG,KM,MS}.  Reaction--diffusion equations cons\-ti\-tute a usual description for complex systems in all these areas. The prototypical model in this context can be written down as
\begin{equation}
\label{plantilla}
\frac{\partial u}{\partial t} = \div \left(D \nabla u \right) + F(u), \quad u(t=0,x)=u_0(x).
\end{equation}
 Here $D$ is a coefficient that  could be a constant (in the simplest case of linear diffusion) \cite{F,KPP,Mein-Shel,Murray}, a function depending on the domain of definition \cite{BH1,BHN,BHN1}, a function depending on $u$ \cite{HR,SG,VA06}, or in general a function $D= D(u, \nabla u)$ \cite{Enguica, Rosenau2}, which includes the possibility of fractional diffusion associated with Levy processes \cite{Caff1,PQRV2}. The function $F$ represents the reaction term. The different models expressed in equation \eqref{plantilla} have been the object of a intense study in the literature in order to clarify : 1) the qualitative differences when we consider  non-linear diffusion operators like the p-Laplacian or the one of the porous media equation in contrast with the behavior associated with a linear diffusion term
 , 2) what kind of particular solutions (as traveling waves, kinks, or solitons, for instance) can be obtained when different functional forms for $F(u)$ are proposed
 , 3) the behavior of systems of equations of type \eqref{plantilla}, or even more complicated instances of them -- this is a way to describe pattern formation and cooperative behavior, see \cite{GG} for instance, 4) the effect of noise on front propagation, see for example \cite{Mu}, 5) the stability or long time asymptotic properties of the patterns, see \cite{BHM,FM,Jo}, for instance.

Our research in this paper falls into the first and second categories above. We analyze the existence of traveling wave solutions associated
to a nonlinear diffusion PDE coupled
to a reaction term of Fisher--Kolmogorov--Petrovskii--Piskunov (FKPP) type \cite{F,KPP}, namely
\begin{equation}\label{modelo1}
\begin{array}{ll}
\displaystyle
\frac{\partial u}{\partial t} = \nu \, \div_x \left(\frac{u^m \nabla u}{\sqrt{|u|^2 + \frac{\nu^2}{c^2}|\nabla u|^2}} \right) + F(u),\hspace{0.3cm} & {\rm in}
\hspace{0.2cm} Q_T=]0,T[\times \R^N,
\end{array}
\end{equation}
where $m > 1$ and $F(u)$ is a Lipschitz continuous function such that $F(0)=F(1)=0$.
 Here $\nu$ is a kinematic viscosity and $c > 0$ is  a characteristic speed \cite{Rosenau2}. Note that \eqref{modelo1} is a renormalization with respect to the carrying capacity $v_0$ of the equation
 \begin{equation}
 \label{carr}
\frac{\partial u}{\partial t} = \nu v_0 \, \div_x \left(\frac{\left(\frac{u}{v_0}\right)^m \nabla_x u}{\sqrt{|u|^2 + \frac{\nu^2}{c^2}|\nabla u|^2}} \right) +  F \left(\frac{u}{v_0}\right).
 \end{equation}
The reaction term in the FKPP case would be given by $F= K v_0 u \left(1-u\right)$, where $K$ is the growth rate.
{Although the results and techniques introduced in this paper  can be extended to more general cases, we focus our attention
on the FKPP case to deal with concrete numerical examples.}

 Equation (\ref{modelo1}) belongs to the class of flux-limited diffusion equations.
Flux limited diffusion ideas were introduced by Rosenau in \cite{CKR, Rosenau2} in order to restore the finite speed of propagation of
signals in a medium. This property is lost in the classical transport theory that
predicts the nonphysical divergence of the flux with the gradient, as it happens also with the
classical theory of heat conduction (based in Fourier's law) and with the linear diffusion theory
(based in Fick's law). Besides Rosenau's
derivation \cite{Rosenau2}, the particular case of (\ref{modelo1}) where $m=1$ was also formally derived by
Brenier by means of Monge--Kantorovich's mass
transport theory in \cite{Brenier1} (this has been done later in a rigorous way in \cite{Mc-Pu}), where he named it as the relativistic heat equation.
 More recently, (\ref{modelo1}) has been shown ($m=1$) to be an effective model
to describe the transport of morphogens in cellular communication to induce distinct
cell fates in a concentration-dependent manner \cite{VGRaS}. 

The model (\ref{modelo1}) (with $F=0$) was introduced in \cite{Rosenau2} (when $m=3/2$) as an example of 
flux limited diffusion equation in the context of heat diffusion in a neutral gas. As shown in \cite{Rosenau2} the acoustic speed 
is a function of the temperature and the front is convected nonlinearly. This has been shown mathematically in
\cite{leysalto} and it has been proved in \cite{PMnew} that solutions of (\ref{modelo1}) converge to solutions of 
the classical porous medium equation as $c\to\infty$. Thus, this model offers a novel dynamical behavior to describe
diffusion and propagation phenomena in real media. The finite propagation property 
is at the basis of this behavior.

Other flux-limited versions of the porous medium equation, namely 
\begin{equation}\label{NPME}
u_t =\nu \mathrm{div}\left(\frac{u\nabla u^m}{\sqrt{1+\frac{\nu^2}{c^2}\vert \nabla u^m\vert}}\right), \qquad \hbox{\rm $m > 0$,}
\end{equation}
were introduced in \cite{Rosenau2,CKR} and further studied in
\cite{PMnewE,PMnew} (where convergence to the classical porous medium equation is proved). They can be also derived using transport theory as proposed in \cite{Brenier1} (see also \cite{PMnew}).
In this case, the acoustic speed is the constant $c$ \cite{PMnewE} and thus independent of $u$.
The different behavior between both types of models is not yet fully understood, but numerical evidence
\cite{CKR,ACMSV,Marquina} shows that model (\ref{modelo1}) may have  a richer behavior creating discontinuity fronts
starting from smooth initial conditions, while the model (\ref{NPME}) exhibits a behavior more similar to the
corresponding standard porous medium equation \cite{CKR}.
Based on this fruitful dynamical behavior, our purpose here is to concentrate on the study of (\ref{modelo1}) leaving
the study of traveling waves for model (\ref{NPME})  for future research.

The dynamics produced by the combined effects of flux limitation and the nonlinearities of
porous media type in (\ref{modelo1}) ($F=0$) 
(nonlinear Fokker--Planck) may be relevant for its potential applicability in the study of  other similar operators.
From a mathematical point of view, this combination
demands the use of different techniques and ideas coming from the fields of nonlinear diffusion
(nonlinear semigroups) and scalar conservation laws, e.g. front propagation and entropy solutions.
This concept of entropy solution determines the geometrical features of the admissible solutions  because
the structure of the singularities that a solution may eventually display is strongly restricted, by virtue of a series of constraints that are ultimately related with the physical principle stating that solutions cannot violate causality. To be more precise, jump discontinuities are characterized by having a vertical profile that moves according to a  Rankine--Hugoniot law.

As we will show, the reaction-diffusion equation (\ref{modelo1})  exhibits new
properties with respect to the classical reaction terms coupled with linear diffusion mechanisms.
The existence of  singular traveling waves is one of these new properties, and it is object of study in this paper.
The construction of such singular patterns requires the development of
novel arguments in dynamical systems. These involve the use of invariant manifolds and blow-up control to analyze the singular phase diagrams associated to the ODE satisfied by traveling wave solutions of
(\ref{modelo1}).
In fact, for some choices of the physical constants the classical theory breaks down and we need to use the concept of entropy solution in the dynamical system context and the
Rankine--Hugoniot jump condition to construct our profiles,
producing discontinuous traveling waves (which, for some particular speed values, may have their support in a half line).
This behavior is reminiscent of shock waves in hyperbolic conservation laws.
Our analysis gives further insight into the properties of the solutions of (\ref{modelo1}) that were experimentally studied in \cite{CKR,ACMSV,Marquina} when $F=0$, in particular on the existence of solutions which are discontinuous  in the interior of their support.

Let us recall that there are several instances of traveling wave solutions not supported in the whole line arising in models with non-linear diffusion mechanisms coupled with reaction terms, among which we mention \cite{aronson1980density,Enguica,Newman80,Sagan,PRE,rosenau2002reaction,SG,SG2,SG3}.
This is an issue of great relevance in several contexts where an infinite speed of propagation of the support
is inconsistent with the experimental observations.
The traveling waves supported on a half line that are constructed in the previous references are all continuous functions. To our knowledge, only the results in \cite{kurganov1997effects,kurganov1998burgers,CGSS} and the ones in this paper are able to produce traveling waves that are not only supported on a half line but also exhibit sharp discontinuity fronts.  The papers \cite{kurganov1997effects,kurganov1998burgers} study traveling waves for models including directly hyperbolic terms of Burger's type
coupled with a diffusion operator of curvature type $\left(\frac{u_x}{\sqrt{1+|u_x|^2}}\right)_x$ and no reaction term. They exhibit the existence of a critical regime above which discontinuous transitions in the traveling wave show up.  The research carried in \cite{Ros-Non} is also related to these issues, as discontinuous steady states,  which are not traveling waves, supported in a half line are obtained out of a reaction-diffusion equation whose diffusion mechanism is very similar to that in \eqref{modelo1},  but avoiding the singularity at $u=0$ (say of curvature type); however the reaction term is not of FKPP type, using cubic and quintic nonlinearities, and the techniques based on numerical and asymptotic methods are different.
Note that the cubic  nonlinearity of the Allen--Cahn term produces a bistability effect, which helps in the study of the unique associated traveling wave. The case of the reaction FKPP term is different as regards the stability of traveling waves and their uniqueness.
Studying these phenomena may open new perspectives of application of these  models to biology or traffic flow frameworks, for instance.


Let us now  introduce our assumptions
on the reaction term $F$.

\subsection{Assumptions on the reaction term}

We will be concerned with the analysis of traveling wave solutions to a family of one-dimensional non-linear flux limited diffusion equations coupled with a reaction term of FKPP type.
Concretely, we are interested in the non-linear diffusion equation which can be written down as (\ref{modelo1}), with $m> 1$. The analysis of such models with $F \equiv 0$ was the object of \cite{ACMMRelat,ACMSV}.
We assume that $F(u)$ satisfies the following properties:
\begin{itemize}
\item $F \in \mathcal{C}^1 ([0,1])$,  $F(0)=F(1)=0$, and $F(u)>0 $ for every $u \in ]0,1[$.
\item $F'(1)<0$.
\end{itemize}
Note that we can write $F(u)=u K(u)$ with
\begin{itemize}
\item $ K\in\mathcal{C}^1(]0,1]) \cap \mathcal{C}^0([0,1])$.
\item $K(1) = 0$ and  $K(u) > 0, \quad \forall u \in ]0,1[$.
\item  $K(0)= F'(0) \geq 0$, $K'(1)= F'(1) <0$.
\end{itemize}
This allows for traveling fronts that connect  the constant state $u=1$ (which, before normalization, would correspond to the state $u=v_0$, see \eqref{carr}) with the zero state. This can be justified by
the comparison principle given in Theorem \ref{UniqSup}, which ensures that we can restrict ourselves to the study of solutions between these two constant states.
We suggest the reader to keep in mind the prototypical case $F(u) = u^p(1-u^q)$, where $p, q\ge 1$ (see \cite{Murray} and references therein for applications). The conditions on function $F$ classify it as a ``Type A'' reaction function according with  the characterization of \cite{BN}.

The next thing we do is to analyze the structure of discontinuous solutions to \eqref{modelo1}. This will make clear what kind of traveling fronts are to be expected.

\subsection{Entropy solutions and the Rankine--Hugoniot condition}\label{sect:analysisEC}

Eq. (\ref{modelo1}) is a particular instance of the class of flux limited  diffusion equations for which the
correct concept of solution, allowing to prove existence and uniqueness results, is the notion of entropy solution \cite{ACMMRelat,CMSV,CEU2}. Although somewhat involved, this notion is necessary since (\ref{modelo1}) (as many other flux limited diffusion equations) has a parabolic-hyperbolic behavior, with solutions that may exhibit moving discontinuity fronts \cite{ACMSV,CKR}. In particular, we notice in passing that the right function space to study this class of solutions is the space of functions of bounded variation.

As usual, the notion of entropy solution of (\ref{modelo1})
is described in terms of a set of inequalities of Kruzhkov type \cite{Kruzhkov} that are well adapted to
prove uniqueness results. But, as proved in \cite{leysalto}   for  $F = 0$, we can give a geometric characterization of
entropy conditions on the jump set of solutions of (\ref{modelo1}). Indeed, in their jump set, entropy solutions of (\ref{modelo1})
have a vertical graph and this is equivalent to the entropy inequalities there. This permits also to give an explicit form to
Rankine--Hugoniot condition that expresses the velocity of moving discontinuity fronts  \cite{leysalto}.
Both things, the geometric characterization of entropy solutions and the Rankine--Hugoniot condition, are relevant for us here.
Indeed, they will guide us in the search for traveling waves of (\ref{modelo1}),
after reducing it to the study of an associated dynamical system (see Section \ref{lados}). Thus, our approach is based on
the analysis of that system, taking into account the properties of entropy solutions of (\ref{modelo1}).

Let us briefly recall both the Rankine--Hugoniot condition and the geometric characterization of entropy solutions of
(\ref{modelo1}) in a context that is suitable for our purposes here.
Since we follow the presentation in \cite{leysalto} we will skip the proofs of the given statements.
For continuity of the presentation, the notation and basic background on the functional setting,
the definition of entropy solutions and basic existence and uniqueness results for
(\ref{modelo1}) are given in the appendix in Section \ref{preliminaris} (see also \cite{Ambrosio}).
Although the case we are interested in here corresponds to $N=1$, let us write them in the general case $N\geq 1$.

Let $Q_T =]0,T[\times \R^N$. Assume that $u \in BV_{\rm loc} (Q_T)$.
Let us denote by $J_u$ the jump set of $u$ as a function of $(t,x)$. For any $t > 0$, we denote by
$J_{u(t)}$ the jump set of $u(t) \in BV_{\rm loc} (\R^N)$.
Let $\nu :=\nu_u =(\nu_t,\nu_x)$ be the unit normal to the
jump set of $u$ so that the jump part of the distributional derivative reads $D^j_{t,x} u = [u] \nu  \H^{N}\vert_{J_u}$, where $ \H^{N}$ is the $N$-dimensional Hausdorff measure in $\R^N$.
We denote by $\nu^{J_{u(t)}}$ the unit normal to the jump set
of  $u(t)$ so that $D_x^ju(t)= [u(t)]\nu^{J_{u(t)}}\H^{N-1}\vert_{J_{u(t)}}$.
Here  $[u](t,x):=u^+(t,x)-u^-(t,x)$ denotes the jump of $u$ at
$(t,x)\in J_u$ and $[u(t)](x):=u(t)^+(x)-u(t)^-(x)$ denotes the jump of $u(t)$ at the point $x\in J_{u(t)}$.

Let us recall the definition of the
speed of the discontinuity set of $u$ \cite{leysalto}.

\begin{definition}\label{def:speed}
Let $u\in BV_{\rm loc}(Q_T)$, $F\in L^1_{\rm loc}(\R^N)$, and let $\z\in L^\infty([0,T]\times\R^N,\R^N)$ be such that
$u_t = \mathrm{div}\, \z + F$ in $\mathcal{D}^\prime(Q_T)$.  We define the speed of the discontinuity set of $u$ as
$v(t,x) = \frac{\nu_t(t,x)}{|\nu_x(t,x)|}$ $\H^N$-a.e. on $J_u$.
\end{definition}

This definition has a sense since, when $u\in BV_{\rm loc}(Q_T)$, $F\in L^1_{\rm loc}(\R^N)$, $\z\in L^\infty([0,T]\times\R^N,\R^N)$ and
$u_t = \mathrm{div}\, \z + F$ in $\mathcal{D}^\prime(Q_T)$, we have (see \cite{leysalto}, Lemma 6.4) that
$$
\H^N\{(t,x)\in J_u: \nu_x(t,x)=0\}=0.
$$

In our next result we state the Rankine--Hugoniot condition in a context that covers the case of equation (\ref{modelo1}).
The proof follows as in  \cite{leysalto} and we omit it.

\begin{proposition}\label{prop:RH1}
Assume that $F:\R\to\R$ is Lipschitz.
Let $u\in BV_{\rm loc}(]0,T[\times \R^N)$ and let $\z\in L^\infty([0,T]\times\R^N,\R^N)$ be such that
$u_t = \mathrm{div}\, \z + F(u)$.
For $\mathcal{L}^1$ almost any $t> 0$ we have
\begin{equation}\label{fla:rhf1}
[u(t)](x) v(t,x)   = [[\z\cdot\nu^{J_{u(t)}}]]_{+-} \qquad \hbox{\rm $\H^{N-1}$-a.e. in $J_{u(t)}$,}
\end{equation}
where $[[\z\cdot\nu^{J_{u(t)}}]]_{+-}$ denotes the difference of traces from both sides of $J_{u(t)}$.
\end{proposition}

We call outer side of $J_{u(t)}$ the one to which
$\nu^{J_{u(t)}}$ is pointing. Thus, the outer trace is $u(t)=u(t)^+$.
Notice that with this notation, the Rankine--Hugoniot condition \eqref{fla:rhf1} is expressed in an
invariant way.
We have denoted as $[\z\cdot\nu^{J_{u(t)}}]$ the weak trace of the normal component of $\z$ on $J_{u(t)}$.
This notion is well defined since $\z$ is a bounded vector field whose divergence is a Radon measure \cite{Anzellotti1,leysalto,ChenFrid1}.
This is covered by the results in \cite{Anzellotti1,ChenFrid1} if $J_{u(t)}$ is locally a Lipschitz surface.
In the present case, we  need the further developments in \cite{leysalto}.

Assume that $m>1$. As in \cite{leysalto}, the notion of entropy solution
of (\ref{modelo1}) (see Section \ref{preliminaris}) can be expressed
as a set of inequalities that can be translated into a geometric condition on the jump set of the solution. Informally, one can say that
jump discontinuities are fronts with a vertical contact angle moving at the speed given by the Rankine--Hugoniot condition. This can be proved
as in \cite{leysalto}.

\begin{proposition}\label{thm:interpEC2}
Let $F:\R\to\R$ be a Lipschitz function.
Let $u \in C([0,T];$ $L^1_{\rm loc}(\R^N))$ $\cap BV_{\rm loc} (]0,T[\times \R^N)$.
Assume that $Du= D^{ac} u + D^j u$, that is, $Du$ has no Cantor part.
Assume that $u_t = \mathrm{div}\, \z + F(u)$ in $\mathcal{D}^\prime(Q_T)$, where $\z=\a(u,\nabla u)$
is the flux of (\ref{modelo1}).
Then  $u$ is an entropy solution of (\ref{modelo1}) if and only if
for $\mathcal{L}^1$-almost any $t > 0$
\begin{equation}\label{eq:vcap}
[\z\cdot \nu^{J_{u(t)}}]_+ =  c (u^+(t))^m  \qquad \hbox{\rm and} \qquad [\z\cdot \nu^{J_{u(t)}}]_- = c(u^-(t))^m
\end{equation}
hold $\H^{N-1}$ a.e. on $J_{u(t)}$.
Moreover, from Proposition  \ref{prop:RH1}, the velocity of the discontinuity fronts is
\begin{equation}\label{eq:speedp}
v = c  \frac{ (u^+(t))^m- (u^-(t))^m}{u^+(t)-u^-(t)}.
\end{equation}
\end{proposition}

To conclude, let us rephrase the conditions \eqref{eq:vcap} in a more geometric way.
Under some additional assumptions they amount to a vertical profile of $u$ on its jump set.
This is the case if we assume that for $\H^N$ almost all $x\in J_u$ there is a
ball $B_x$ centered at $x$ such that either (a) or (b) hold, where

\begin{itemize}
\item[(a)] $u\vert_{B_x} \geq \alpha > 0$,
\item[(b)] $J_u \cap B_x$ is the graph of a Lipschitz function with $B_x \setminus J_u = B_x^1 \cup B_x^2$,
where $B_x^1,B_x^2$ are open and connected,
$u\geq \alpha > 0$ in $B_x^1$, while the trace of $u$ on $J_u \cap \partial B_x^2$ computed from $B_x^2$ is zero.
\end{itemize}

In both cases, under the assumptions of Proposition  \ref{thm:interpEC2}, by Lemma 5.6 in \cite{leysalto} we can cancel $u^m$ on both sides of
the identities in (\ref{eq:vcap}) and obtain
\begin{equation}\label{verticalA1}
\left[\frac{\nabla u}{\sqrt{u^2+\frac{\nu^2}{c^2} \vert \nabla u\vert^2}}\cdot \nu^{J_{u(t)}}\right]_+ = \frac{c}{\nu}  \qquad \hbox{\rm on $J_u\cap B(x,r)$.}
\end{equation}
If (a) holds we also have
\begin{equation}\label{verticalA2}
\left[\frac{\nabla u}{\sqrt{u^2+\frac{\nu^2}{c^2} \vert \nabla u\vert^2}}\cdot \nu^{J_{u(t)}}\right]_- = \frac{c}{\nu}  \qquad \hbox{\rm on $J_u\cap B(x,r)$.}
\end{equation}
In dimension one, assuming that the jump point is isolated and  that $u$ is smooth out of the discontinuity, the above conditions mean  that the graph of $u$ is vertical at the discontinuity points. The same can be said in any dimension if $J_{u}$ is a regular surface and
$u$ is smooth out of the discontinuity set. In the more general case, the traces in (\ref{verticalA1}), (\ref{verticalA2})
are interpreted in a weak sense \cite{Anzellotti1,leysalto,ChenFrid1}.

We conclude this first section by introducing  the main results of this paper.

\subsection{Statement of the main results}

We assume that $m > 1$.
We look for one-dimensional incoming wave solutions of (\ref{modelo1}) with range in $[0,1]$, traveling at constant speed $\sigma>0$, with their shape being completely unaltered. That is, we look for solutions of the form $u(x-\sigma t)$. In a first step we will study decreasing traveling profiles,
but we will prove that monotonicity is not a real constraint because these are the only piecewise smooth entropy solutions of (\ref{modelo1}) having a traveling wave structure.
Let us make precise that when we say that a function is piecewise smooth, up to a finite number of points, we understand
that at those singular points there is a jump either of the function or of its first derivative.
For this type of solutions our main result  (see Fig. \ref{TWprofiles}) is the following:

\begin{theorem}
\label{parto}
Let $m > 1$. The following results are verified

\begin{itemize}
\item[i)] Existence: There exist two values $0<\sigma_{ent}<\sigma_{smooth}<mc$, depending on $c,\nu, m$ and $F$, such that:
\begin{enumerate}
\item for $\sigma>\sigma_{smooth}$ there exists a unique smooth traveling wave solution of
\eqref{modelo1},
\item for $\sigma=\sigma_{smooth}$ there exists a traveling wave solution of \eqref{modelo1}, which is continuous but not smooth,
\item for $\sigma_{smooth} > \sigma \ge \sigma_{ent}$ there exists a 
 traveling wave solution of \eqref{modelo1}, which is discontinuous.
\end{enumerate}

\item[ii)] Uniqueness:  for any fixed value of $\sigma \in [\sigma_{ent},+\infty[$, after normalization (modulo spatial translations) there is just one traveling wave solution in the class of piecewise smooth solutions (that is, smooth except maybe at a finite number of points)  with range in $[0,1]$ and satisfying the entropy conditions.

\item[iii)]Continuity: Assume that there is a value $p\ge 1$ such that
\begin{equation}\label{palota}
\liminf_{u \to 0}\frac{F(u)}{u^p} = k \in ]0, +\infty].
\end{equation}
After suitable normalization, there is a family of traveling wave solutions $u^{\mathcal N}_{\sigma}$ for $\sigma \in [\sigma_{ent},+\infty[$ which enjoys the following property:
$$
\lim_{\sigma_1 \to \sigma_2} \|u^{\mathcal N}_{\sigma_1}(t)-u^{\mathcal N}_{\sigma_2}(t)Ê\|_{L^p(\R)} +   \|u^{\mathcal N}_{\sigma_1}(t)-u^{\mathcal N}_{\sigma_2}(t)Ê\|_{L^\infty (\R)}= 0,
$$
for any $t \geq 0$ and any $\sigma_1, \sigma_2 \in [\sigma_{ent},+\infty[$.
\end{itemize}
\end{theorem}
\begin{remark}
Note that condition (\ref{palota}) has been introduced to complement the uniform convergence with a convergence in some $L^p$ space. Other possible assumptions could involve different spaces.
Let us point out that when the condition (\ref{palota})
holds for some $p\ge 1$, then it is also verified for any value of $p$ above this one. When $F$ is analytic, we can take $p$ as the order of the zero of $F$ at $u=0$. In the general case, we may not be able to find a minimal value of $p$ such that (\ref{palota}) is fulfilled, an example of this situation being given by $F(u) = u^2  \log \left(\frac{1}{u}\right)$. In any case, the convergence will be at least uniform, see Section 3.

\end{remark}

From the perspective of applications, the most interesting and novel solutions are those
corresponding to $ \sigma \in [\sigma_{ent},\sigma_{smooth}[$, which are discontinuous.
In particular, those corresponding to $\sigma = \sigma_{ent}$ are supported on a half line for each $t$, and they encode processes in which the propagation of information (whatever it may be) takes place at finite speed.

%

\begin{figure}[h]
\begin{center}
\includegraphics[width=13cm]{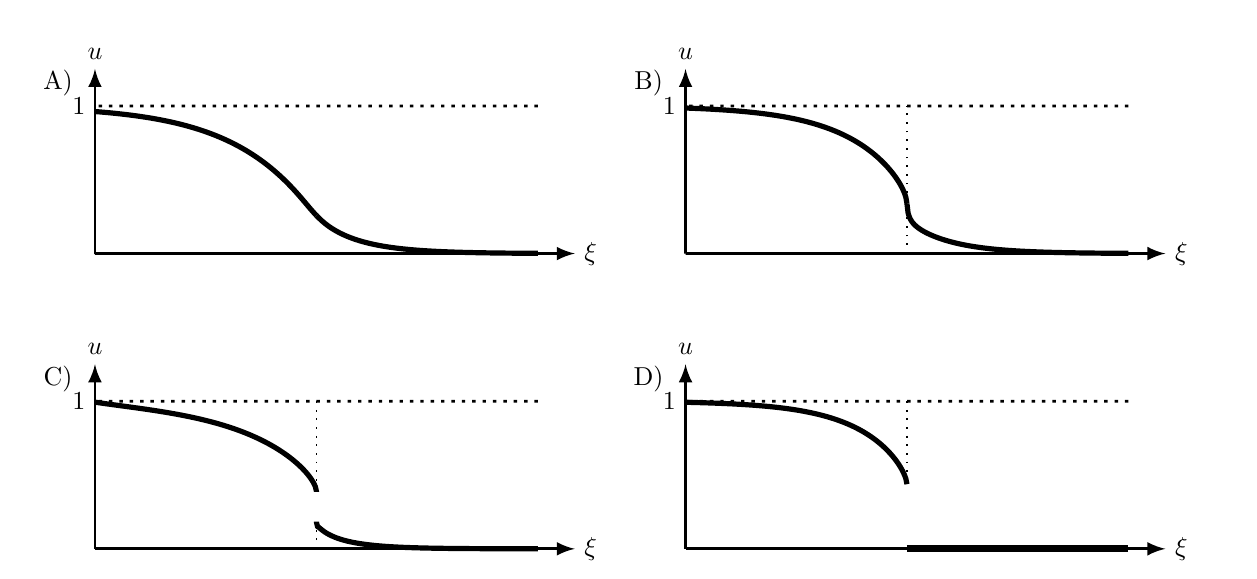}
\caption{Traveling wave profiles for \eqref{modelo1}, see Theorem \ref{parto}:  {A)}  $\sigma > \sigma_{smooth}$,  {B)}  $\sigma = \sigma_{smooth}$, {C)}  $\sigma \in ]\sigma_{ent}, \sigma_{smooth}[$, {D)}  $\sigma = \sigma_{ent}$. Vertical dotted lines show points with infinite slope.
These profiles are in correspondence with the orbits depicted in Fig. \ref{orbitas} {B)}.
}\label{TWprofiles}
\end{center}
\end{figure}

Let us remark that for $m=1$ the catalog of possible traveling wave solutions is restricted to cases A) and D) in Figure \ref{TWprofiles}.
This was discussed  in \cite{CGSS}  by studying the corresponding dynamical system, which has a simpler behavior than for $m > 1$.

\subsection{Plan of the paper}
Let us finally explain the plan of the rest of the paper. In Section \ref{lados} we reduce the study of traveling wave solutions of (\ref{modelo1}) to the study of an associated dynamical system. We prove the existence of different types of
maximal solutions for this system, depending on the speed of the traveling wave, which will be smooth for high enough wave speeds and
discontinuous when the wave speed lowers below a certain threshold. In Section \ref{sect4} we combine the different wave
profiles obtained in Section \ref{lados} and construct the traveling waves of (\ref{modelo1}). Moreover, we also prove uniqueness
of piecewise smooth traveling waves with a given speed, and their continuous dependence on it. We also give some numerical insights about how the  patterns can be attractors of time-dependent solutions and how saturation of diffusion in competition with reaction originate shocks.
Finally, in Section
\ref{preliminaris} we provide an appendix with  the necessary background on entropy solutions of (\ref{modelo1}) in order to give sense to the
statements in this Introduction.

\section{The associated planar dynamical system}
\label{lados}

In order to construct traveling wave profiles we substitute the traveling wave ansatz $u(x- \sigma t)$ into \eqref{modelo1}. This leads to the study of the following equation:
\begin{equation}
\label{start}
\nu \left(\frac{u^m u'}{\sqrt{u^2 + \frac{\nu^2}{c^2}|u'|^2}} \right)' + \sigma u' + F(u) = 0 \qquad \hbox{\rm in $\mathcal{D}^\prime(\R)$.}
\end{equation}
We can use (\ref{start}) to construct piecewise smooth entropy solutions of (\ref{modelo1}). For that, it suffices to join
together smooth solutions of (\ref{start}) defined on intervals of $\R$ fulfilling the following rules: 
\begin{itemize}
\item[(i)] If solutions corresponding to two consecutive intervals match in a
continuous
way, then the first derivative cannot have a jump discontinuity. Otherwise the term $\left(\frac{u^m u'}{\sqrt{u^2 + \frac{\nu^2}{c^2}|u'|^2}} \right)'$ would contribute with a Dirac delta  at
    the matching point, while the terms $\sigma u' + F(u)$ would be in $L^1_{\rm loc}(\R)$, and (\ref{start}) could not hold in
    $\mathcal{D}^\prime(\R)$. The same argument shows that when two solutions match in a continuous way and the first derivative is $+\infty$ (resp. $-\infty$) on one side then it must be also $+\infty$ (resp. $-\infty$) on the other side.
\item[(ii)] If solutions corresponding to two consecutive intervals match forming a jump discontinuity, then the speed of the moving front should obey the Rankine--Hugoniot condition (\ref{eq:speedp}) and the slope of the profile at both sides of the discontinuity must be infinite with the same sign (see \eqref{verticalA1}--\eqref{verticalA2}), except when one of the solutions we are matching with is the zero solution. In that case, when looking for decreasing profiles, we only have to worry about the infinite slope condition on the left side of the discontinuity.
\end{itemize}

In order to search for smooth solutions of (\ref{start}) in intervals of $\R$, we write (\ref{start}) as  an  autonomous planar system. For that we set
$$
r(\xi) = - \frac{\nu}{c} \frac{u'(\xi)}{\sqrt{|u(\xi)|^2 + \frac{\nu^2}{c^2}|u'(\xi)|^2}}.
$$
When looking for decreasing profiles, we observe that $r(\xi)\in [0,1]$ for all $\xi\in\R$ (while $r(\xi)\in [-1,1]$ for all $\xi\in\R$ if no monotonicity assumption is made).
Moreover, if $u(-\infty)=1$, $u(+\infty)=0$ and $u$ is smooth, then $u(\xi)\in [0,1]$ for all $\xi\in\R$.
Then,  for smooth solutions,  \eqref{start} is equivalent to the following  first order planar dynamical system:
 \begin{equation}
 \label{singular}
\displaystyle  \left\{
\begin{array}{cl}
 u'=  & \displaystyle - \frac{c}{\nu} \frac{r u}{\sqrt{1-r^2}},
\\
 &
 \\
\displaystyle r'=  & \displaystyle \frac{1}{u^{m-1}} \frac{r}{\sqrt{1-r^2}} \left(m u^{m-1}\frac{c}{\nu} r - \frac{\sigma}{\nu} \right) + \frac{F(u)}{c u^m}.
\end{array}
\right.
\end{equation}
In what follows only decreasing traveling profiles will be studied, for these are the only reasonable traveling waves that can be obtained, as we show in forthcoming Proposition \ref{unicidad}. Thus, through the present Section we restrict the study of \eqref{singular} to the set $[0,1]\times [0,1]$; this will be implicitly assumed in every statement referring to \eqref{singular}.  We notice that the flux related to the previous system is singular at the boundaries $r=1$ and $u=0$.
The first difficulty that we meet is precisely to give a sense to \eqref{singular} at those points. Indeed,
it will turn out that solutions of \eqref{singular} eventually hit either $u=0$ or $r=1$.
Thus, we will start considering solutions defined in $0<r<1$ and $0<u<1$, which give rise to smooth (classical) traveling wave solutions of \eqref{start} in intervals of $\R$. Then, entropy solutions of (\ref{modelo1}) can be constructed by pasting those solutions
while satisfying rules $(i)$ and $(ii)$ above. If the solutions of (\ref{start}) are defined in all $\R$, they are smooth entropy solutions.

\begin{remark} We note that the change of variables above does not coincide with the standard one in this type of problems.
\end{remark}

In the next two   subsections  we analyze the planar system (\ref{singular}). The knowledge of  the
Rankine--Hugoniot  relation  \eqref{eq:speedp} will be crucial to match solutions of (\ref{singular}) producing
discontinuous profiles that satisfy the entropy conditions.

\subsection{The blow-up sets of the planar system}
The following characterization of the clustering points of the orbits solving \eqref{singular} constitutes a key result in order to analyze the behavior of such orbits.

\begin{proposition}
\label{ab}
Let $(u,r):]\omega_-,\omega_+[\rightarrow ]0,1[\times]0,1[$ be a maximal solution of \eqref{singular} with $\sigma>0$. Then it satisfies  the following:
\begin{enumerate}

\item if $\omega_-=-\infty$, then $\lim_{\xi \to \omega_-} (u(\xi),r(\xi)) = (1,0)$,

\item if $\omega_+=+\infty$, then $\lim_{\xi \to \omega_+} (u(\xi),r(\xi)) = (0,r^*)$,

\item if $\omega_+<+\infty$, then $\lim_{\xi \to \omega_+} (u(\xi),r(\xi)) = (u^+,1)$ for some  $u^+ \in[u^*,1[$,

\item if $\omega_- > -\infty$, then the limit of $(u(\xi),r(\xi))$, as $\xi \to \omega_-$, belongs to one of the sets $]0, u^*] \times \{1\}$, $]0,1[ \times \{0\}$, or $\{1\} \times ]0,1[$.
\end{enumerate}
The points $(u^*,1)$ and $(0,r^*)$ are defined by
\begin{equation*} \label{u*}
u^*:=u^*(\sigma) = \left(\frac{\sigma}{c m} \right)^\frac{1}{m-1}
\end{equation*}
and
\begin{equation*} \label{r*}
 r^*:=r^*(\sigma) = \frac{\frac{\nu  K(0)}{c \sigma}}{\sqrt{1 + \left( \frac{\nu  K(0)}{c \sigma}\right)^2}}.
\end{equation*}
\end{proposition}
\begin{proof}
We start the proof by justifying the existence of the limits $\lim_{\xi \to \omega_\pm} (u(\xi),$ $r(\xi))$ $= (u^\pm,r^\pm)$   for any  solution of \eqref{singular}, then we deal with the four specific assertions of the proposition. We do this in a series of steps.

{\it Step 1.} Using standard arguments on continuation of solutions of an ODE, it is straightforward to deduce that the pairs $(u^\pm,r^\pm)$  should belong to the boundary of $[0,1]\times[0,1]$ in case they exist. To show that these limits exist we pass to an equivalent system which is absent of singularities. This is achieved  formally multiplying both equations in \eqref{singular} by $u^{m-1} \sqrt{1-r^2}$. Thus, we end up with a system  on $]0,1[\times ]0,1[$ which is not singular,
\begin{equation}
 \left\{
\begin{array}{cl}
 U'=  & - \frac{c}{\nu} R U^m,
\\
 &
 \\
R'=  & R \left(m U^{m-1}\frac{c}{\nu} R - \frac{\sigma}{\nu} \right) + \frac{K(U)}{c}\sqrt{1-R^2}.
\end{array}
\right.
\label{regu}
\end{equation}
Solutions of \eqref{regu} are related to solutions of \eqref{singular} by means of
 $r(\xi) = R(\phi(\xi))$, $ u(\xi) = U(\phi(\xi)), $
 where $\phi$ is an strictly increasing reparametrization governed  by $\phi'(\xi) = \frac{1}{u(\xi)^{m-1}\sqrt{1-r^2(\xi)}}$.
 The analysis of the directions of the flux on the boundaries of the $(U,R)$-domain is the same as the one we would perform for the $(u,r)$-system, but having the advantage that the flux is continuous in
$[0,1]\times [0,1]$.

{\it Step 2.} Existence for the initial (and final) value problem for \eqref{regu} is granted in the whole closed set. The regularity of the flow ensures uniqueness in the set $]0,1] \times [0,1[$. We also have uniqueness in the set $\{0\}\times[0,1[$. This is seen as follows: First, if an orbit verifies that $U(\xi_0)=0$ for some $\xi_0 $ in its domain, then it is easily seen using the first equation of \eqref{regu} that $U=0$ in its whole domain of existence. Next, we notice that for every such orbit the second equation in \eqref{regu} gives the value of $R'$ as a smooth function of $R$ alone, thus we have uniqueness of solutions for it.

{\it Step 3.}
Note that $r^*$ and $u^*$ appear when we study the  equilibria and bouncing points --see below-- of the $(U,R)$-system.
In fact, the points $(0, r^*)$ and $(1,0)$ are equilibria. In addition, the flow in $]0,1[\times ]0,1[$  points to the left, except at the boundaries, where we have the following flux analysis (see Fig. \ref{orbitas}-A):
\begin{enumerate}
\item If $R=0$, $U \in ]0,1[$,  the flux is completely vertical and pointing inwards.

\item If $U=1$, $R\in ]0,1[$, the flux is always pointing inwards.

\item If $R=1$, there are two possibilities  depending on the value of $\sigma$.
If $\sigma \ge c m$ the flux is always heading SW; notice that $u^*\ge 1$ in such a case. On the other hand, if $\sigma < c m$ and $U\in ]0,u^*[$ the flux points SW while for $U \in ]u^*,1[$ the flux points NW.

\item If  $U=0$, we have a positively invariant manifold.
For $R\in ]0,r^*[$ the flux is completely vertical and pointing upwards. While for $R \in ]r^*,1[$ the flux is again completely vertical but pointing downwards. Solutions of \eqref{regu} constrained to this manifold are globally attracted by $(U=0,R = r^*)$.

\item The point $(u^*, 1)$ is a regular bouncing point in the sense that the vector field is horizontal and pointing to the left. At this level of discussion we  do not have tools to precise if $(u^\pm, r^\pm)$ could be identified with $(u^*,1)$ for some solution. As we will  specify later,  these possibilities can appear for some types of solutions.
\end{enumerate}

{\it Step 4.} If $(u,r)$ is a solution of \eqref{singular} defined in $]\omega_-, \omega_+[$, then the monotone change of variables $\phi : \, ]\omega_-, \omega_+[ \to ]\omega_-^{(U,R)}, \omega_+^{(U,R)}[$ allows to obtain $(U(\phi(\xi)),R(\phi(\xi))) $ $= (U(\xi'), R(\xi'))$, which is a solution of \eqref{regu} in the interval $]\omega_-^{(U,R)}, \omega_+^{(U,R)}[$.
Now we show that the limits $\lim_{\xi ' \to \omega_\pm^{(U,R)}} (U(\xi'), R(\xi'))$ exist. This is immediate for the component $U(\xi')$ as it is monotone. Thanks to our knowledge of the flux diagram at the boundaries and its continuity on $[0,1]\times [0,1]$, we are able to rule out wild oscillations of the orbits close to their hypothetical clustering points, thus the existence of $\lim_{\xi ' \to \omega_\pm^{(U,R)}} R(\xi')$ follows easily.

Then we can ensure that
$$
(u^\pm, r^\pm) = \lim_{\xi \to \omega_\pm} (u(\xi), r(\xi))  = \lim_{\xi ' \to \omega_\pm^{(U,R)}} (U(\xi'), R(\xi')).
$$
From the previous  flux analysis (see Fig. \ref{orbitas}-A) we know that $(u^-, r^-) \in ]0,1] \times \{ 0\} \cup ]0, u^*] \times \{1\} \cup \{1\} \times ]0,1[$. The event $(u^-, r^-) =(0,r^*)$ cannot take place since $u$ is decreasing. Note also that $(u^+, r^+) \in \{(0,r^*)\} \cup [u^*,1[ \times \{1\}$. In the same way as above $(u^+,r^+)=(1,0)$ is excluded since $u$ is decreasing.
We also observe that no solution starting at any point in $]0,1[\times ]0,1[$ can reach the point $(1,1)$ --it is never an exit point. It only can be an entrance
point when $\sigma \geq mc$. Then, it will be
always considered  in the entrance set.

{\it Step 5.}
Now we are ready to prove the precise assertions of the proposition. We start with the first one. To begin, we note that  $(u^-,r^-)$ does not belong to $]0,1[ \times \{0\} \cup \{1\} \times ]0, 1[ $ and that the flow \eqref{singular} at $(1,0)$ is regular. To show that $(u^-, r^-) \notin ]0,u^*] \times \{1\}$ we  will argue by contradiction (having proved that, the first assertion follows). For this purpose, we can use the monotonicity of $u$ and the mean value theorem to construct a sequence $\xi_n \to -\infty$ for which $u'(\xi_n) \to 0$ (see \cite{CGSS} for details). This contradicts the fact that
$$
\lim_{n \to +\infty} -\frac{r(\xi_n)u(\xi_n)}{\sqrt{1-r(\xi_n)^2}}=-\infty.
$$
Thus we have $\omega^- \neq - \infty$ and the first assertion is verified.

To prove the second assertion, it is enough to remark that $(u^+, r^+) \notin [u^* ,1[ \times \{1\}$. This can be proved by a similar contradiction argument as in the previous case but taking here
a sequence $\xi_n \to \infty$ for which $u'(\xi_n) \to 0$.

The third assertion follows if we can prove that
$\lim_{\xi \to \omega_+} (u(\xi),r(\xi)) = (0,r^*)$ cannot hold for $\omega_+<+\infty$.
Integrating for $u(\xi)$ in \eqref{singular} leads us to
\begin{equation}
\label{base}
u(\xi) = u(0) \exp \left\{-\frac{c}{\nu}\int_0^\xi \frac{r(s)\ ds}{\sqrt{1-r^2(s)}}Ê\right\}.
\end{equation}
If we are to have $u(\omega^+)=0$ for some $\omega^+<+\infty$, then we need the above integral to be divergent for $\xi = \omega^+$.
 But this cannot happen as $r(\xi)$ tends to $r^*<1$ when  $\xi$ goes to $\omega^+$.

Finally, to prove the fourth assertion we have to exclude the case $(u^-,r^-) =(1,0)$. But this is an easy consequence of the fact that $(1,0)$  is  a regular equilibrium of \eqref{singular}.
 \end{proof}
 \begin{remark}
 Note that when  $\sigma = 0$ we have that $r^*=1$ and then the proof of Proposition \ref{ab} breaks down, namely \eqref{base} is no longer useful. However, in this case the behavior of the orbits is simpler
as we will show in Proposition \ref{zero} below.
 \end{remark}
\begin{remark}
In the proof of Proposition \ref{ab} we have introduced an auxiliary regular system in order to analyze the direction fields in $[0,1]\times [0,1]$. Note that the cases $r=1$ and $u=0$ are singular for  \eqref{singular}, but we can analyze the direction fields on these sets as ``limits'' of those in the regular regions. These direction fields coincide with those of  the regular system \eqref{regu}, see Figure \ref{orbitas}-A.
\end{remark}
\begin{figure}[h]
\begin{center}
\includegraphics[width=13cm]{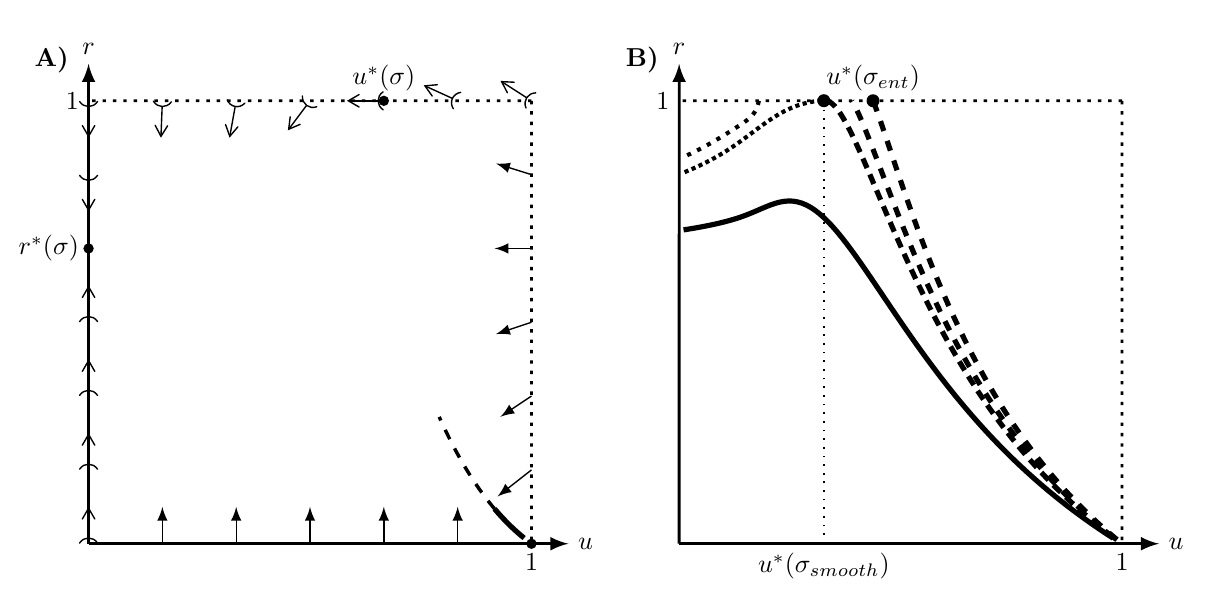}
\caption{{A)} Normalized direction field of the flux related to \eqref{singular} for $\nu = c = 1$,  $m=2$, $F(u) = u(1-u)$ and $\sigma = 4/3$.
Different types of arrows were used to stress the fact that the actual flux is singular at the boundaries $r=1$ and $u=0$. {B)} Numerical solutions to Type I (solid), II (dashed) and III (dotted) orbits of \eqref{singular} for several values of $\sigma$; see  Definitions \ref{tipos12} and \ref{tipo3} for an explanation of the terminology. The parameters are the same as those on A); with that choice $\sigma_{ent} =0,437803$ and $\sigma_{smooth}=0,661621$. The lowest Type II and III orbits are those corresponding to $\sigma_{smooth}$. The uppermost Type II orbit corresponds to $\sigma_{ent}$. The intermediate Type II orbit and the uppermost Type III orbit correspond to a value $\sigma \in ]\sigma_{ent},\sigma_{smooth}[$ and are related by  the jump law \eqref{eq:speedp}.   The values $\sigma_{ent}$ and $\sigma_{smooth}$ are defined in Propositions \ref{s1} and \ref{propdisc}. }\label{orbitas}
\end{center}
\end{figure}
\subsection{Solutions to the planar system defined on a half line}
\label{campos1}

Since we are looking for entropy solutions to \eqref{modelo1} which are piecewise smooth we can discard
all those orbits of \eqref{singular} defined in bounded intervals. By Proposition \ref{ab} all of them exhibit
a finite slope at least at one of the ends of their interval of definition, which, as we mentioned  at the beginning of Section 2, makes them useless in order to construct globally defined solutions by means of matching procedures (see however the proof of Proposition \ref{unicidad} for a more detailed explanation of this fact). So,  in this section we will deal with solutions globally defined in the whole $\R$, or with solutions
defined in  a half line (that is $]-\infty,\omega_+[ $ or $]\omega_-,+\infty[$) and such that their slope at $\omega_\pm$ is infinite.

The following theorem describes all those orbits defined in $\R$ or $]-\infty,\omega_+[ $.

\begin{theorem}
\label{caaampos}
For each value of $\sigma \ge 0$ there exists, up to reparametrization, a unique solution of \eqref{singular}  defined in $]-\infty,\omega_+[ $ with $\omega_+\leq\infty$. Furthermore, this solution satisfies:
\begin{equation}
\label{condini}
\lim_{\xi \to -\infty} (u(\xi),r(\xi)) = (1,0)
\end{equation}
and $\lim_{\xi \to -\infty} u'(\xi) = 0$. Finally, we have that
\begin{equation}
\label{pendientes}
\lim_{\xi \to -\infty} \frac{r'(\xi)}{u'(\xi)}=\lambda_\sigma = \frac{1}{c} \frac{2 \nu K'(1)}{\sigma + \sqrt{\sigma^2 - 4 K'(1) \nu}}<0.
\end{equation}
\end{theorem}
\begin{proof}
 Let us observe that if any such orbit is to exist, then it has to verify (\ref{condini}), thanks to Proposition \ref{ab}. The linearization of the system near $(u=1,r=0)$ gives the following Jacobian matrix:
 $$
\displaystyle   \left(
 \begin{array}{cc}
0 & \displaystyle -\frac{c}{\nu} \\
\displaystyle \frac{K'(1)}{c} & \displaystyle -\frac{\sigma}{\nu}
 \end{array}
 \right).
 $$
 After computing its eigenvalues
$$
\lambda_\sigma^\pm=-\frac{\sigma}{2 \nu} \pm \frac{1}{2} \sqrt{\frac{\sigma^2}{\nu^2}- \frac{4 K'(1)}{\nu}},
$$
we learn that this point is hyperbolic. This allows to apply the unstable manifold theorem (see \cite{Hartman}). Since the eigenvector associated with $\lambda_\sigma^+$ is $(c \frac{\sigma + \sqrt{\sigma^2-4 \nu K'(1)}}{2\nu K'(1)},1)$, the solutions starting at $(1,0)$  enter the diagram,  and they do it with a slope given by \eqref{pendientes}. Thanks to the Hartman--Gro{\ss}man theorem, these solutions are uniquely determined for a given value of $\sigma$.
\end{proof}

\begin{remark}
\label{decre}
The lower the value of $\sigma$, the higher the entrance angle ({measured with respect to the $u$-axis}). There is a maximum value for $\sigma = 0$, namely
$$
\lambda_0 = -\frac{\sqrt{\nu |K'(1)|}}{c}.
$$
We also stress that $\lambda_\sigma$ is increasing as a function of $\sigma$.
\end{remark}

At this point we introduce some terminology that will be useful in the sequel, see Fig \ref{orbitas}. We name the particular types of trajectories we will be interested in.

\begin{definition}
\label{tipos12}
Let $(u(\xi),r(\xi)) \in ]0,1[\times ]0,1[$ be a maximal solution  of \eqref{singular}--\eqref{condini} for a given value of $\sigma$.
We shall say that:
\begin{itemize}

\item  $(u,r)$ is  a Type I orbit  if $(u,r)$ is defined in $\R$ and $\lim_{\xi \to +\infty} (u(\xi),r(\xi)) = (0,r^*(\sigma))$

\item  $(u,r)$ is  a Type II orbit  if  $(u,r)$ is defined in $]-\infty,\omega_+[$, with $\omega_+ < \infty$, and $\lim_{\xi \to \omega_+} (u(\xi),r(\xi)) = (u^+,1)$ for some $0<u^+<1$.

 \end{itemize}

 For a given value of $\sigma$ only one of these two possibilities occurs.
 \end{definition}
Note that no Type II orbit can show up for $\sigma \ge mc$, since $u^*(\sigma)\ge 1$.

\begin{remark}
\label{prop}
Any Type I orbit induces a smooth solution to \eqref{start} and satisfies $\lim_{\xi \to +\infty} u(\xi)=0$. Such profiles are depicted in Figure \ref{TWprofiles} A). Type II  orbits can be found  in Figure \ref{TWprofiles} too; they correspond to the left  branches  (with respect to the vertical dotted line)  in the cases B, C and D.
\end{remark}
We stress that
all profiles coming from Type I orbits are regular solutions of \eqref{singular} supported in the whole line. Type II  orbits may also give rise to traveling wave solutions, after a suitable matching procedure to extend them to the whole real line is performed. This will be explained in Section \ref{sect4}, but before that we need the following:

 \begin{definition}
 \label{tipo3}
We will say that a maximal solution  $(u(\xi),r(\xi))$ of \eqref{singular} for a given value of $\sigma$
is a Type III orbit if  $(u,r)$ is defined  in  $]\omega_-,+\infty[$, $\lim_{\xi \to \omega_-} (u(\xi),$ $r(\xi)) = (u^-,1)$ for some $0<u^- \leq u^*$ and $\lim_{\xi \to +\infty} (u(\xi),r(\xi)) = (0,r^*(\sigma))$.
\end{definition}

As regards the uniqueness of Type II and III  orbits with respect to the beginning or ending point (in the limit sense established in Proposition \ref{ab}) we will describe the orbits of the planar system as graphs $u \mapsto r(u)$ whenever this is possible.
This is always the case if we are prepared to allow some derivatives to become infinite eventually (and this may happen only at the boundaries of the domain). Indeed, if a trajectory can be expressed locally as a graph $u \mapsto r(u)$, its derivative is given by
\begin{equation}
\label{u-grafo}
\frac{d r(u)}{d u} = \frac{r'}{u'} = \frac{\sigma}{c u^m} - m \frac{r}{u} - \frac{K(u)}{u^m} \frac{\nu}{c^2} \frac{\sqrt{1-r^2}}{r}.
\end{equation}

\begin{lemma}\label{lemma4}
The formulations \eqref{singular} and \eqref{u-grafo} are equivalent. Moreover, {regular} solutions to \eqref{singular}--\eqref{condini} correspond to solutions of \eqref{u-grafo} such that
\begin{equation}
\lim_{u \to 1}r(u)=0 .
\label{bc}
\end{equation}
In addition, $\lim_{u \to 1}r'(u)=\lambda_\sigma$ holds.
\end{lemma}
\begin{proof}
This is straightforward once we notice that $u'< 0$  in the  domain.
\end{proof}

Once we are able to pass to the formulation given by \eqref{u-grafo}, we get the following result.

\begin{lemma}
\label{viejo}
Existence and uniqueness for \eqref{u-grafo} holds backwards at any point of the form $(\tilde{u},1)$ with $ u^*(\sigma) \leq \tilde{u}<1$. Existence and uniqueness for \eqref{u-grafo} holds forwards at any point of the form $(\tilde{u},1)$ with $0 < \tilde{u} \leq u^*(\sigma)$ and $\tilde{u}<1$.
\end{lemma}
\begin{proof}
The existence problem is addressed by solving the initial value problem for \eqref{u-grafo} with $r(\tilde u) = 1$, being $\tilde u \neq u^*(\sigma)$. To do this we will use {Peano's existence theorem}, and for that we need a continuous  extension of \eqref{u-grafo} to values $r>1$.

In the case $\tilde u > u^*$ we have $r'(\tilde u) <0$ and, therefore, the function $r$ defined on $]\tilde u, \tilde u + \epsilon[$, for some small $\epsilon>0$, maps into $]0,1[$, solving our original problem. This solution can be extended to $]\tilde u, 1[$ and it verifies $r(u) <1$; otherwise we could find a first value $u_0 > \tilde u$ such that $r(u_0)=1$ and $r'(u_0) <0$, which would give us a contradiction.

In a similar way, for  $\tilde u < u^*$ we find a solution of \eqref{u-grafo} on $]\tilde u -\epsilon, \tilde u[$ such that
$r'(\tilde u) >0$, which  can be extended to $]0, \tilde u[$.

The case $\tilde u = u^*(\sigma)$ can be treated by approximation. In fact, taking a sequence $\tilde u_n \to u^*(\sigma)$, $ \tilde u_n \neq u^*(\sigma)$ and depending on either $\tilde{u}_n<u^*$ or $\tilde{u}_n>u^*$, we find a solution on $]0, u^*(\sigma)[$ or on $]u^*(\sigma), 1[$. Note that the approximating sequence has a  partial subsequence which is convergent, via Peano's theorem on continuous dependence with respect to initial conditions and parameters. The limit verifies $r(u) \leq 1$ either on the case of solutions in $]0, u^*(\sigma)[$ or on the case of solutions in $]u^*(\sigma), 1[$. Let us analyze this last case: a solution of the extended problem \eqref{u-grafo} {cannot be equal to 1  on an interval having $u^*(\sigma)$ as its left end, since this is not coherent with the values of the flux defined by \eqref{regu} at the boundary}. On the other hand, the fact that $r(u) <1$ for some value $u$ implies that $r(u) <1$ for greater  values and thus for $u \in ]u^*(\sigma), 1[$.

In every case, the uniqueness follows  from the change of variable
$\sqrt{1-r^2(u)}= s(u)$ leading to the following differential equation:
\begin{eqnarray*}
s^{\prime }
&=&\frac 1s\left( -\frac \sigma {cu^m}+\frac mu\right) -\frac{\sigma \left(
-s^2\right) }{cu^ms\left( \sqrt{1-s^2}+1\right) }-\frac{ms}u+\frac{k(u)}{u^m}%
\frac \nu {c^2} \\
&=&\frac 1s\left( -\frac \sigma {cu^m}+\frac mu\right) +h(u,s),
\end{eqnarray*}
where $h(u,s)$ is a Lipschitz function in the second  variable in a neighborhood of   $(u, s=0)$, and
$s\rightarrow \frac 1s\left( -\frac \sigma {cu^m}+\frac mu\right) $ is a  decreasing function
if $u>u^{*}(\sigma)$, resp. increasing
if $u<u^{*}(\sigma)$.
We conclude by the classical uniqueness results for equations with right hand side given by a Lipschitz part  plus a monotone
part \cite{Hartman}.
\end{proof}

The following result allows to construct Type III orbits which start at any point of $]0, u^*(\sigma)] \times \{1\}$.
\begin{proposition}
For any $u^-\in ]0, u^*(\sigma)]$, there exists a Type III orbit that satisfies \eqref{singular} and such that
\begin{eqnarray}
\lim_{t \to \omega_-}  (u(t), r(t)) = ( u^-,1).
\label{3.1}
\end{eqnarray}
Moreover, this  orbit  is unique up to reparametrizations.
\label{viejo3}
\end{proposition}

\begin{proof}
Let $u^- \in ]0, u^*(\sigma)]$. Consider the solution  of \eqref{u-grafo} such that
\begin{equation}
r(u^-)=1 ,
\label{3.2}
\end{equation}
and   let $(r_0, u_0)$ be  a point on the graph of $r$, with $u_0 \in ]0,u^-[$. It is easy to see  that the solution of
\begin{equation} \label{r-sigma}
u' = - { \frac{c}{\nu} \frac{r(u) u}{\sqrt{1-r(u)^2}}}
\end{equation}
verifying $u(t_0)= u_0$, for any given $t_0 \in \R$, supplemented with $r(t) =r(u(t))$, is a solution of \eqref{singular} which   blows up  in $\omega_- > -\infty$ (singular in $u^-$) and  which satisfies \eqref{3.1}. 
We have also that any Type III orbit verifying \eqref{3.1}  is a reparametrization of the above one. Thus,  a solution $(r,u)$ of \eqref{singular}--\eqref{3.1} is a Type III orbit such that $u'(t) <0$ in $]\omega_-, + \infty[$. Therefore, we can  deduce that $u$ is a diffeomorphism with its image, which necessarily is $]0, u^-[$. Then, we can invert it and take $r \circ u^{-1} \, : ]0, u^-[ \to ]0,1[$, which verifies \eqref{u-grafo}--\eqref{3.2}. Hence, $r \circ u^{-1} \equiv r(t)$ and, as a consequence, $u$ is a solution of \eqref{r-sigma}. This allows to assure that $u$ differs from the previous solution  just by a reparametrization.  The same happens with $r$, as it is obtained from $u(t)$.
\end{proof}

\subsection{Bifurcation from Type I to Type II orbits} 
\label{sec4}

The goal of this section is to analyze the structure of discontinuous solutions of \eqref{modelo1}, in terms of the special classes of orbits defined in Section \ref{campos1}. Our first aim is to describe the set of values of $\sigma$ for which these singular solutions can be constructed.

\begin{definition}
Let
$$
\Sigma_{I}= \{\sigma\ge 0 \ : \mbox{ the associated  solution of }  \eqref{singular}\!\!-\!\!\eqref{condini}
\mbox{ is a Type I orbit}\}.
$$
 Also, let
$$
\Sigma_{II}= \{\sigma\ge 0 \ : \mbox{ the associated  solution of }  \eqref{singular}\!\!-\!\!\eqref{condini}
\mbox{ is a Type II orbit}\}.
$$
\end{definition}
\noindent
Note that $\Sigma_{II}$ is bounded from above by $mc$.

According to Lemma \ref{lemma4},
for every $\sigma  \in \Sigma_{I}$  the corresponding solution of  \eqref{singular}--\eqref{condini}
derives from a solution $r_\sigma : ]0,1[ \to ]0,1[$ of  (\ref{u-grafo})--(\ref{bc})
 satisfying
$$
\lim_{u \to 0} r_\sigma(u) = r^*.
$$
When $\sigma \in \Sigma_{II}$, we consider the escape point $u^+(\sigma)$  of  \eqref{singular}--\eqref{condini} through $r=1$
as introduced in Definition \ref{tipos12}. Then,
 $r_\sigma : ]u^+(\sigma),1[ \to ]0,1[$  is a solution of  (\ref{u-grafo})--(\ref{bc}).
The function $u^+$ is, therefore, defined from $\Sigma_{II}$ to $]0,1[$ and verifies $u^+(\sigma) \geq u^*(\sigma)$.
The way to recover the solutions of  \eqref{singular}--\eqref{condini}
 from the solutions of  (\ref{u-grafo})--(\ref{bc})
 is to integrate the differential equation \eqref{r-sigma}
 as in Proposition \ref{viejo3}.

Now we show that the orbits which are candidates for representing traveling wave profiles are ordered with respect to $\sigma$.
\begin{lemma}
\label{compara}
If $\sigma_1 < \sigma_2$, then $r_{\sigma_1}(u)>r_{\sigma_2}(u)$ in their common domain of definition.
\end{lemma}
\begin{proof}
Recall that $\lambda_\sigma$ increases strictly as $\sigma$ increases (Remark \ref{decre}). Then $r_{\sigma_1}(u) > r_{\sigma_2}(u)$ in a neighborhood of $u=1$. If our thesis is false then there exists a first value $0<\tilde{u}<1$ such that $r_{\sigma_1}(\tilde{u})= r_{\sigma_2}(\tilde{u})$. Then, the hypothesis $\sigma_1 < \sigma_2$ implies by  \eqref{u-grafo} that $\frac{dr_{\sigma_1}}{du}(\tilde{u})<\frac{dr_{\sigma_2}}{du}(\tilde{u})$, which constitutes a contradiction.
\end{proof}

The above result implies that $\Sigma_{I}$ and  $\Sigma_{II}$ are intervals, i.e., given a value of $\sigma$ such that the corresponding orbit is  a Type I  (resp. Type II) orbit, then this is also the case for upper (resp. lower) values of $\sigma$; we have also that $u^+(\sigma)$ is a decreasing function of $\sigma$.  Moreover, $\Sigma_{I}  \cup \Sigma_{II} =[0,+\infty[$ (it is in fact a Dedekind cut) and in the sense of sets $\Sigma_{II}  < \Sigma_{I}$, i.e., if $\sigma \in \Sigma_{II}$, then $\Sigma_{II}$ contains all values below this $\sigma$, $[0,  \sigma] \subset \Sigma_{II}$ and  if $\sigma \in \Sigma_{I}$, then $\Sigma_{I}$ contains all values above this $\sigma$, $[\sigma,  +\infty[ \subset \Sigma_{I}$; moreover all the elements of $\Sigma_{II}$  are below those of $\Sigma_{I}$.
Note that at this stage we have not ruled out the possibility of having $\Sigma_{II}=\emptyset$ yet. The following  result takes  care of this issue.
\begin{proposition}
\label{zero}
The {maximal} orbit associated with $\sigma=0$ verifying \eqref{u-grafo}--\eqref{bc} exits the phase diagram by a point $(u^+(0),1)$ with $u^+(0)>0$.\end{proposition}
\begin{proof}
The equation \eqref{u-grafo} can be recast after multiplication by $u^m$ as
$$
u^m r'(u) +m u^{m-1}r(u) = -\frac{\nu}{c^2} K(u) \frac{\sqrt{1-r^2}}{r}.
$$
This implies
$$
\frac{d}{du}(u^m r(u)) \le 0.
$$
Assume now that $r(u)$ is defined for $u\in ]0,1[$. Then $\lim_{u\to 0} u^m r(u) =0$ and thus $u^m r(u)$ is identically equal to zero, which contradicts the fact that the slope at $u=1$ is known to be strictly negative. This shows that $r(u)$ is defined only in an interval $]u^+(0),1[$ and that $\lim_{u \to u^+(0)}r(u)=1$.
\end{proof}
The previous result shows that $\Sigma_{II}$ is not empty, at least it contains the value zero. In the next Proposition we prove that it contains a non-trivial interval of values and we characterize its supremum.

\begin{proposition}
\label{s1}
The value $\sigma_{smooth}=\sup \{\sigma: \sigma \in \Sigma_{II}Ê\}$ verifies  $mc>\sigma_{smooth}>0$.  Moreover:
\begin{enumerate}

\item any maximal solution satisfying \eqref{singular}--\eqref{condini} with $\sigma > \sigma_{smooth}$ is a Type I orbit,

\item any maximal solution satisfying \eqref{singular}--\eqref{condini}  with $\sigma \le \sigma_{smooth}$ is a Type II orbit.

\end{enumerate}
Furthermore, $\sigma_{smooth}$ is the unique value of $\sigma$ having the property that the associated solution, which is a   Type II orbit, terminates at the point $(u^*(\sigma),1).$ Then $u^+(\sigma_{smooth}) = u^* (\sigma_{smooth})$.

\end{proposition}
\begin{proof}
We notice that for $\sigma>\sigma_{smooth}$ we get Type I orbits, while for $\sigma < \sigma_{smooth}$ we get Type II orbits. By Remark \ref{prop} we also know that $\sigma_{smooth}< mc$.
It only remains to prove that $\sigma_{smooth} >0 $, that this value belongs to $\Sigma_{II}$, and to verify that $u^+(\sigma_{smooth})=u^*(\sigma_{smooth})$.
These claims follow from the continuous dependence of solutions of \eqref{singular}--\eqref{condini} w.r.t $\sigma$.  More precisely,

 \begin{lemma}
\label{cont}
Consider the maximal solutions of \eqref{u-grafo} extended to the right end by means of $r_{\sigma_n}(1)=0$. The following assertions are satisfied:
\begin{enumerate}
\item  Let $\{\sigma_n\}_{n\geq 0}$ be a monotonically decreasing sequence such that $\sigma_n\rightarrow \sigma_\infty$, with  $\sigma_{\infty} \in \Sigma_{II}$.
Then the sequence $\{r_{\sigma_n}\}$ of  solutions to (\ref{u-grafo})--(\ref{bc})  converges uniformly on compact sets of  $]u^+(\sigma_{\infty}),1]$ to $r_{\sigma_\infty}$. Moreover, if
$\sigma_n \in \Sigma_{I}$, then  $u^+(\sigma_{\infty})=u^*(\sigma_{\infty})$,  and if $\sigma_n \in \Sigma_{II}$ for advanced $n$, then $u^+(\sigma_{n}) \to u^+(\sigma_{\infty})$.
\item
 Let $\{\sigma_n\}_{n\geq 0}$ be a monotonically increasing sequence such that $\sigma_n\rightarrow \sigma_\infty$, with $\sigma_n \in \Sigma_{II}$.
 Then, $\sigma_\infty \in \Sigma_{II}$  and $u^+(\sigma_{n}) \to u^+(\sigma_{\infty})$. In addition, the sequence $r_{\sigma_n} : ]u^+(\sigma_n) , 1] \to ]0,1]$ converges on compact sets of $]u^+(\sigma_\infty) , 1]$ to $r_{\sigma_\infty}$.

\end{enumerate}
\end{lemma}
\begin{proof}
As we pointed in the statement of the lemma, in this proof we will consider the functions $r_\sigma$ to be extended by continuity to their value at $u=1$, even  if the differential equation \eqref{u-grafo} is defined only on the open interval, being  singular at $ r = 0 $. We split the proof into two steps.

{\it Step 1.} Let us start by proving the second assertion. We stress that, due to Lemma \ref{compara} and to the monotonicity of $\sigma_n$, the sequence $r_{\sigma_n}$ has increasing intervals of definition $]u^+(\sigma_n), 1]$. Define
$
\alpha = \inf_{n \in \NN} \left\{ u^+(\sigma_n) \right\};
$ our aim is to prove that $u^+(\sigma_\infty)$ is well defined and coincides with $\alpha$.

As a consequence of Lemma \ref{compara} we have
\begin{equation}\label{jj1}
\left\{
\begin{array}{ll}
\mbox{either} & r_{\sigma_\infty}(u) \mbox{ is defined on } u \in ]0,1] \\ & \\
\mbox{or} & u^+(\sigma_\infty) \leq \alpha.
\end{array}
\right.
\end{equation}
Moreover, given $u \in ]\alpha,1]$, from  Lemma \ref{compara} we deduce that the value $r_{\sigma_n}(u)$ is defined for $n \in \NN$ large enough. Furthermore, these values constitute a decreasing sequence. Then, we  define
\begin{equation*}\label{jj2}
\left\{
\begin{array}{l}
\tilde  r \, : \, ]\alpha, 1] \to ]0,1] \\  \\
\tilde  r (u) = \displaystyle \lim_{n \to \infty} r_{\sigma_n}(u).
\end{array}
\right.
\end{equation*}
The alternative \eqref{jj1} implies that the domain of any $ r_{\sigma_n}$ is contained in that of $\tilde  r$. Using again Lemma \ref{compara} we obtain
\begin{equation}\label{jj3}
r_{\sigma_\infty}(u) \leq \tilde r(u), \quad \forall u \in ]\alpha,1] .
\end{equation}
Now,  the inequality $u^+(\sigma_n) \geq u^*(\sigma_n)$ leads to
\begin{equation}\label{jj4}
\alpha \geq u^*(\sigma_\infty) > 0.
\end{equation}
We also have the following estimate which is independent of $n$
\begin{equation}\label{jj5}
\displaystyle \frac{K(u)}{r_{\sigma_n(u)}} \ \leq \ \sup_{u \in ]\alpha, 1[}\ \ \frac{K(u)}{r_{\sigma_\infty}(u)}, \qquad \forall u \in ]u^+(\sigma_n),1[.
\end{equation}
Note that the function $u \in ]\alpha, 1[ \mapsto \frac{K(u)}{r_{\sigma_\infty}(u)}$ is bounded at both endpoints --recall that $K'(1)$ exists-- and thus on its whole domain.

Combining \eqref{jj4}, \eqref{jj5} and \eqref{u-grafo} we deduce
\begin{equation}\label{jj6}
|r'_{\sigma_n}(u) | \leq M,  \quad \forall u \in ]u^+(\sigma_n),1[,
\end{equation}
where $M$ is independent of $n$.
Then, Ascoli's theorem ensures that $r_{\sigma_n}$ converges to $\tilde r$ uniformly on compact sets of $]\alpha, 1]$; the bound \eqref{jj6} is also valid at the boundary once all the objects are properly extended. In particular, $\tilde r$ is a continuous function verifying \eqref{bc}.
At the same time, \eqref{jj3} ensures  that $r'_{\sigma_n}$ also converges uniformly on compact sets of $]\alpha,1[$. Hence, $\tilde r$ satisfies \eqref{u-grafo} and thanks to Lemma \ref{lemma4} we have $\tilde r \equiv r_{\sigma_\infty}$ on $]\alpha, 1]$.

Finally, let us prove that $\alpha = u^+(\sigma_\infty)$. For that, it suffices  to check that
\begin{equation}\label{jj7}
\lim_{u \to \alpha} \tilde r(u) =1.
\end{equation}
To do this, let $\epsilon > 0$ and let $u \in ]\alpha, \alpha + \epsilon[$. Then for  $n$ large enough
we have that $| r_{\sigma_n}(u) - \tilde r(u) | < \epsilon$. Since also $u^+(\sigma_n) > \alpha$ for $n$ large enough, we have that $| u - u^+ (\sigma_n)| \leq \epsilon$. Then, the mean value theorem and estimate \eqref{jj6} yield that $|r_{\sigma_n}(u) -1| \leq M\epsilon$. Since $u\in ]\alpha, \alpha + \epsilon[$ we find $|\tilde r (u) -1| \leq \epsilon + M\epsilon$, and \eqref{jj7} follows. This finishes the proof of the second assertion.

{\it Step 2.} We are now concerned with the proof of the first assertion.  Due to Lemma \ref{compara} the sequence of functions $\{ r_{\sigma_n} \}_{n \in \NN}$ is defined on a common interval: either  $]\alpha , 1]$ with  $\alpha = \sup_{\sigma_n \in \Sigma_{II}} \left\{ u^+(\sigma_n)\right\} >0 $ or $]0,1]$ when $\sigma_n \in\Sigma_{I}$ for any $n\in \NN$.

In the first case, the same argument as in the previous step leads to
\begin{equation*} \label{jj8}
\frac{K(u)}{r_{\sigma_n}(u)} \leq  \sup_{u \in ]\alpha, 1[} \frac{K(u)}{r_{\sigma_1}(u)},
\end{equation*}
which allows to prove the uniform convergence of $r_{\sigma_n}$  to a function $\tilde r$ on  $]\alpha, 1]$. Note that the uniform convergence on compact sets of the sequence $r'_{\sigma_n}$ allows to deduce that   $\tilde r$ verifies \eqref{u-grafo}; this is possible because $\tilde r(u) \geq r_{\sigma_1}(u) >0$, $\forall u\in ]\alpha,1[$. This function $\tilde r$ coincides (because it is a pointwise limit) with $r_{\sigma_\infty}$ on $]u^+(\sigma_\infty), 1[ $ since in this case Lemma \ref{compara} ensures that
\begin{equation} \label{jj9}
\alpha \leq u^+ (\sigma_\infty) .
\end{equation}
The case $\alpha \neq u^+ (\sigma_\infty)$ can be excluded by a contradiction argument. In fact, $\tilde r$ is a $C^1$ function such that $\tilde r(u^+(\sigma_\infty)) =1$ and $\tilde{r}\le 1$. Thus, this value is a maximum and, therefore, $\tilde r'(u^+(\sigma_\infty)) =0$. Taking into account \eqref{u-grafo} and
\[
\tilde r'(u^+(\sigma_\infty)) = \lim_{u \to u^+(\sigma_\infty)} \tilde r'(u)
\]
we find $u^+(\sigma_\infty) = u^*(\sigma_\infty)$. Therefore,
$\lim_{n \to \infty} u^*(\sigma_n) \leq \alpha= \lim_{n \to \infty} u^+(\sigma_n) <  u^+(\sigma_\infty) = u^*(\sigma_\infty),
$
which is in contradiction with \eqref{jj9}.

It remains to study the case $\sigma_n \in \Sigma_{I}$, for any $n \in \NN$. Then, we use a value $\alpha \in ]0, u^+(\sigma_\infty)[$ to argue in a similar way as we did above and we finish once we find that $u^+(\sigma_\infty) = u^*(\sigma_\infty)$.
\end{proof}
\noindent
{\it  (Proof of Proposition \ref{s1}, continued)}

 Assertion {\it 1)} in Lemma \ref{cont} together with Proposition \ref{zero} prove that $\sigma_{smooth} > 0$,
 because  $u^*(\sigma=0)<u^+(\sigma=0)$. Then we can apply
 Lemma \ref{cont}.{\it (2)} to any sequence that converges to $\sigma_{smooth}$ and we conclude that this value belongs to $\Sigma_{II}$.
 Finally,  $\sigma_{smooth}$ is the only value $\sigma$ such that $u^*(\sigma)=u^+(\sigma)$. This follows from  Lemma \ref{cont}.{\it (1)} and the facts that $\sigma \mapsto u^+(\sigma)$ is strictly decreasing (see the paragraph after Lemma \ref{compara}) and  $\sigma \mapsto u^*(\sigma)$ is strictly increasing.
 \end{proof}
At this point we can use
Proposition \ref{viejo3} to show that the orbit associated to $\sigma_{smooth}$ can be extended as a continuous curve further to the right (matching with a Type III orbit). More is true, as we show in our next result, which is paramount in order to characterize completely the discontinuous traveling wave solutions of \eqref{modelo1}.

\begin{proposition}
\label{propdisc}
There exists a value $0 < \sigma_{ent} < \sigma_{smooth}$ such that the
following assertions hold true in the range $ \sigma_{ent} \le \sigma
\le \sigma_{smooth}$:
\begin{enumerate}

\item Any Type II orbit can be extended to the whole $\R$ matching it
with a Type III orbit.

\item  There is only one way to perform the aforementioned matching. It is given by the following formula:
\begin{equation}
\label{sentcond}
\sigma = c \frac{(u^+(\sigma))^m -(u^-(\sigma))^m}{u^+(\sigma) -
u^-(\sigma)}.
\end{equation}
Here $(u^+(\sigma),1)$ is the arrival point for the Type II orbit and
$(u^-(\sigma),1)$ is the departure point for the Type III orbit.

\item Moreover $\sigma \mapsto u^-(\sigma)$ is a  continuous, strictly
increasing mapping, and the value $\sigma_{ent}$ is defined as the value
of $\sigma \geq 0$ for which
$$
\lim_{\sigma\to \sigma_{ent}}u^-(\sigma)=0.
$$
In addition, we have
$$
\lim_{\sigma\to
\sigma_{smooth}}u^-(\sigma)=u^+(\sigma_{smooth})=u^*(\sigma_{smooth}).
$$
\end{enumerate}
\end{proposition}

The proof of this result is just a consequence of the previous ideas, together with Proposition \ref{thm:interpEC2} and the following statement.

\begin{lemma}
\label{luegotellamo}
Let $\sigma \le \sigma_{smooth}$. Then, if \eqref{sentcond} is fulfilled, there must hold that $\sigma>\sigma_{smooth}/m$. Whenever \eqref{sentcond} holds, the pair $(u^+(\sigma),u^-(\sigma))$ is unique and the mapping $\sigma \to u^-(\sigma)$ is strictly increasing. Finally, \eqref{sentcond} holds at least for a neighborhood $]\sigma_{smooth}-\epsilon,\sigma_{smooth}]$ of $\sigma_{smooth}$.
\end{lemma}
\begin{proof}
For the sake of clarity  we will denote $u^+(\sigma), \ u^-(\sigma)$  by $u^+, \ u^-$ whenever this creates no confusion.
 Given the value $u^+$, we want to figure out the value of $u^-$ in order that  \eqref{sentcond} holds.
 It may happen that no such value exists. To deal with this issue,
we consider the continuous function
$$
\psi(u^+, x)=
\left\{
\begin{array}{lcl}
\frac{(u^+)^m - x^m}{u^+ - x} & \mbox{if}  & x \neq u^+ \\ & & \\
m (u^+)^{m-1} & \mbox{if}  & x = u^+
\end{array}
\right.
$$
 defined for $x\in [0,\infty[$. The first thing to note is that \eqref{sentcond} is trivially satisfied for $\sigma=\sigma_{smooth}$ with $u^+=u^-$. Note that once $u^+$ is fixed $\psi$ is a strictly increasing function, since
 $$
\frac{\partial \psi(u^+,x)}{\partial x}= \frac{(u^+)^m}{(u^+-x)^2} \left( (m-1) \left(\frac{x}{u^+}\right)^m - m \left(\frac{x}{u^+}\right)^{m-1} +1 \right) >0 \quad \mbox{for}\ x \neq u^+\, .$$
Thus, $\psi(u^+,\cdot)$ is a bijection,  $\psi(u^+,\cdot) : [0,u^+] \to [ (u^+)^{m-1}, m (u^+)^{m-1}]$. We must check if $\sigma/c$ belongs to the latter interval. As $u^+>\left( \sigma/(mc)\right)^\frac{1}{m-1}$ we deduce that
$$
\frac{\sigma}{c} < m (u^+)^{m-1}.
$$
It remains to be determined when do we have that $(u^+)^{m-1} \le \sigma/c$. Notice that for $\sigma = \sigma_{smooth}$ the above inequality is strict. Thus, it continues to hold for some neighborhood $]\sigma_{smooth}-\epsilon,\sigma_{smooth}]$, thanks to Lemma \ref{cont} (more precisely, we know that  the value of $u^+$ increases as $\sigma$ decreases).
The previous arguments ensure that in such a case there is a unique pair $(u^+,u^-)$ verifying \eqref{sentcond}.
We can also prove that the mapping $\sigma \mapsto u^-(\sigma)$ is  strictly  increasing, because  if this is not the  case the existence of values $\sigma_1 < \sigma_2$ such that $u^-(\sigma_1) \ge u^-(\sigma_2)$ leads to the following contradiction:
$$
\frac{\sigma_1}{c} = \psi(u^+(\sigma_1),u^-(\sigma_1)) > \psi(u^+(\sigma_2),u^-(\sigma_2)) = \frac{\sigma_2}{c},
$$
where we have used that the function $\psi$ is increasing in both variables (by symmetry)  and the fact that $u^+(\sigma_1) > u^+(\sigma_2)$.

Now we show that $u^-\le u^*$, so that $(u^-,1)$ can be a departure point for a Type III orbit. More precisely, either $u^+=u^*=u^-$ or $u^+>u^*>u^-$.  To show that, we write
$$
(u^+)^m-(u^-)^m = \int_{u^-}^{u^+} m s^{m-1}\ ds.
$$
Under any of the events $u^+\ge u^- > u^* $ or $u^+>u^-\ge u^* $ we have
$$
(u^+)^m-(u^-)^m> m (u^*)^{m-1} (u^+-u^-).
$$
Then we learn that $\sigma/c> m (u^*)^{m-1}=\sigma/c$, which constitutes a contradiction. This implies that $u^+>u^*>u^-$ or $u^+=u^*=u^-$.
Finally, the necessary condition $\sigma>\sigma_{smooth}/m$ shows up at once, since
$\sigma_{smooth}/mc = (u^+(\sigma_{smooth}))^{m-1}$  and thus the relation $u^+(\sigma)^{m-1}\le \sigma/c$ (which was seen to be required in order that an admissible choice of $u^-$ exists) cannot hold for $\sigma=\sigma_{smooth}/m$, being the map $\sigma \mapsto u^+(\sigma)$ strictly decreasing.
\end{proof}

\begin{remark}
When $m=2$ the condition \eqref{sentcond} reduces to
$$
u^-= \frac{\sigma}{c} - u^+.
$$
\end{remark}

\begin{remark}
Estimates so far show that
$$
\sigma_{smooth}/m<\sigma_{ent}<c \ \mbox{and}\ \sigma_{ent}<\sigma_{smooth}<mc.
$$
This is coherent with the case $m=1$ \cite{CGSS}.
\end{remark}

\section{Construction of traveling wave solutions}
\label{sect4}
The purpose of this section is to prove Theorem \ref{parto}.
Let us first precise that our solutions 
satisfy the  property of having null flux at infinity. First of all, we have proved in Theorem \ref{caaampos} that $\lim_{\xi \to -\infty}  u'(\xi) =0$, but it is also true that $\lim_{\xi \to \infty}  u'(\xi) =0$ because $\frac{u'(\xi)}{u(\xi)} \to -\frac{K(0)}{\sigma}$ holds, 
as $\xi \to \infty$, and $r^*\leq 1$. Then
$$
\lim_{|\xi| \to \infty} \frac{u'(\xi)u^m(\xi)}{\sqrt{|u(\xi)|^2 + \frac{\nu^2}{c^2}|u'(\xi)|^2}}=0,
$$
and our claim follows. Thanks to our study of dynamical system \eqref{singular} we have all the tools required to describe the traveling wave solutions of \eqref{modelo1}.  This is the object of our next results.
\begin{proposition}
\label{propES}
The following statements hold true:
\begin{enumerate}

\item  Any Type I orbit induces a smooth traveling wave $u(x-\sigma t)$  which is an entropy solution of \eqref{modelo1}
with null flux at infinity (see Definition \ref{def:nullflux} in Appendix \ref{SectSSS}).
Hence they are unique in the sense of the initial value problem, with initial condition $u(x)$.
This is the case for $\sigma > \sigma_{smooth}$.

\item For any $\sigma_{ent} \leq \sigma \leq \sigma_{smooth}$, there exists a traveling wave solution $u(x-\sigma t)$ with null flux at infinity. These traveling waves are
unique entropy solutions in the sense of the initial value problem, with initial condition $u(x)$. Moreover:
\begin{itemize}

\item
When $\sigma_{ent} < \sigma<\sigma_{smooth}$ the traveling wave is discontinuous at the
junction $x-\sigma t = 0$ and smooth off of it.
The slope is infinite at both sides of this point.

\item When $\sigma=\sigma_{smooth}$ the traveling wave is continuous in the
whole line,
$x-\sigma t =\xi\in\R$,  and smooth off of the junction at  $x-\sigma t
= 0$.
The slope is infinite at both sides of this point.

\item  If  $\sigma = \sigma_{ent}$, then  $u^-(\sigma)=0$ and the corresponding solution is supported on a half line $x-\sigma t = \xi \in \R^-$.
\end{itemize}
 \end{enumerate}
\end{proposition}

\begin{proof}
$(1)$ is a consequence of Remark \ref{prop} and Proposition \ref{s1}. The uniqueness result follows from
Theorem \ref{UniqSup} (see Appendix \ref{SectSSS}).

 $(2)$ When $\sigma < \sigma_{smooth}$ we have that $u\in C([0,T];L^1_{\rm loc}(\R^N)) \cap BV_{\rm loc}(]0,T[\times \R^N)$
and $Du$ has no Cantor part. Since by Proposition \ref{propdisc} the speed of the discontinuity fronts satisfies \eqref{eq:speedp},
then Proposition \ref{thm:interpEC2} implies  that $u(x-\sigma t)$ is an entropy solution of \eqref{modelo1}. As a concatenation of
Type II and Type III orbits, it is smooth out of the discontinuity set and has a null flux at infinity. When
$\sigma = \sigma_{smooth}$, the traveling wave satisfies $u\in C([0,T];L^1_{\rm loc}(\R^N)) \cap W^{1,1}_{\rm loc}(]0,T[\times \R^N)$. Hence,
 by Proposition \ref{thm:interpEC2}, it
is an entropy solution. As a concatenation of
Type II and Type III orbits, it has a null flux at infinity. Uniqueness  follows from
Theorem \ref{UniqSup} (see  Appendix  \ref{SectSSS}). The additional statements are consequences of Proposition \ref{propdisc}.
\end{proof}

Now we wonder about the number of traveling waves that can be constructed with a given speed.
Let us recall that when we say that a function is piecewise smooth, up to a finite number of points, we understand
that at those singular points there is a jump either of the function or of its first derivative.

\begin{proposition}
\label{unicidad}
Given any $\sigma \in [\sigma_{ent},+\infty[$, the only nontrivial entropy
 solution of \eqref{modelo1} with the form $u(x -\sigma t)$, having its range in $[0,1]$ and  being
piecewise smooth --up to a finite number of points-- is (up to
spatial shifts) the one provided by Proposition \ref{propES}.
\end{proposition}

\begin{proof}
The proof is divided into a series of steps.

\emph{Step 1. Precise setting of the problem.} Let $\sigma \in [\sigma_{ent},+\infty[$ and let $u(x-\sigma t)$ be a
traveling wave
which is piecewise smooth, up to a finite number of points, and satisfies the entropy conditions. Recall that, as it was shown in Section \ref{lados}, traveling wave solutions of \eqref{modelo1} with range in $[0,1]$ are in close correspondence with solutions of the system \eqref{singular} considered over the range $0\le u \le 1,\ -1\le r\le 1$ (here we are not making any monotonicity assumption). Thus, during this proof we consider the system \eqref{singular} to be defined on $[0,1]\times [-1,1]$.

Let $I_i=]\xi_i,\xi_{i+1}[$, $i=1,\ldots,p$,  be  maximal intervals of
smoothness
of $u$, so that $\xi_1=-\infty$, $\xi_{p+1}=+\infty$, and either $u$ or $u'$ has a jump
point at $\xi=\xi_i$ for all $i=2,\ldots,p$.
Since $Du$ has no Cantor part, entropy solutions are characterized by
Proposition \ref{thm:interpEC2} and so observations $(i)$ and $(ii)$ in
Section \ref{lados} hold. Moreover, $u$ is a solution of (\ref{start})
in $\mathcal{D}^\prime(\R)$ and the pair
$(u(\xi),r(\xi))$ is a solution of (\ref{singular}) in each interval $I_i$.

\emph{Step 2.}  We show that each of the intervals $I_i$ is a maximal interval of existence for the system \eqref{singular} (that is, cutting and matching at will does not yield reasonable solutions). Let $i$ be a fixed value. As $I_i$ is a maximal interval of smoothness, then it is a subset of a maximal interval of existence of \eqref{singular}. Assume for instance that the maximal interval of existence has the form $]\xi_{i},\overline{\xi}[$ for some $\overline{\xi}>\xi_{i+1}$, the other possibilities can be handled in a similar way. Then,
there exists a smooth pair $(\overline{u}(\xi),\overline{r}(\xi))$ defined on $]\xi_{i},\overline{\xi}[$ as a maximal solution to \eqref{singular}, such that $(u,r)$ and $(\overline{u},\overline{r})$ coincide over $I_i$, but $(u,r)(\xi_{i+1}^+) \neq (\overline{u},\overline{r})(\xi_{i+1}^+)$. As $r(\xi_{i+1}^-)=\overline{r}(\xi_{i+1}^-)$ we get that $u'(\xi_{i+1}^-)$ is finite. Being $u(x-\sigma t)$ an entropy solution of \eqref{modelo1}, in case that $u(\xi_{i+1}^+) = \overline{u}(\xi_{i+1}^+)$ we must have that $u'(\xi_{i+1}^+) = \overline{u}'(\xi_{i+1}^+)\in \R$ thanks to observation (i) at the beginning of Section \ref{lados}; thus $u(\xi)$ could be extended smoothly to the right of $I_i$, which would be a contradiction. Then, this means that $u(\xi_{i+1}^+) \neq \overline{u}(\xi_{i+1}^+)$. Knowing that $u(x-\sigma t)$ is an entropy solution of \eqref{modelo1}, observation (ii) at the beginning of Section \ref{lados} shows that $|u'(\xi_{i+1}^-)|=\infty$, but this is again a contradiction as we already knew that this value was finite. Thus, the only way out is to conclude that $I_i$ is a maximal interval of existence.

\emph{Step 3.}  Let us prove that $(u(\xi),r(\xi))$ $ \to (1,0)$ as $\xi\to-\infty$.
Proceeding as in Proposition \ref{ab}, the analysis of the flow given by \eqref{singular} at the boundaries of $[0,1]\times [-1,1]$  is straightforward. This can be combined with arguments similar to those in Proposition \ref{ab} to show that either $(u(\xi),r(\xi))$ tends to $\{u=0\}\times [-1,1]$ when $\xi\to-\infty$ or $(u(\xi),r(\xi))$ $ \to (1,0)$ as $\xi\to-\infty$.

The first possibility yields only the zero solution: note that the set $\{u=0\}\times [-1,1]$ is positively invariant under the flow \eqref{singular}. So, no attempt to try to construct a non-trivial solution such that $(u(\xi),r(\xi))$ tends to $\{u=0\}\times [-1,1]$ as $\xi\to-\infty$  is successful. Indeed, any such solution would be equal to zero in $I_1$, with $\xi_2 <+\infty$, and being not identically zero we have to extend it further to the right in a non-trivial way. Being $\{u=0\}\times [-1,1]$ positively invariant under the flow, the only way to do this is performing a discontinuous matching with some other orbit defined in $I_2$. The matching to be performed has to satisfy the requirements set up in Proposition \ref{thm:interpEC2}, which implies that the profile must be traveling from right to left, i.e. $\sigma <0$. This contradicts the assumptions of the current proposition. Thus, the only chance that is left is to have $(u(\xi),r(\xi))$ $ \to (1,0)$ as $\xi\to-\infty$.

\emph{Step 4.} By Theorem \ref{caaampos}  and our assumption on the range of the traveling wave, the solution $(u,r)$ in $I_1$ is unique and
satisfies that
$(u(\xi),r(\xi))$ $ \to (1,0)$ as $\xi\to-\infty$. The solution has a
decreasing profile in $I_1$ and a limit
$u(\xi_2^-)$ as $\xi\to \xi_2^-$. By Proposition \ref{s1}, if $\sigma >
\sigma_{smooth}$, then $\xi_2 = +\infty$,
and $u$ is smooth in all $\R$ and coincides with the solution
constructed in Proposition \ref{propES}.
If $\sigma \in [\sigma_{ent},\sigma_{smooth}]$, then $u$ is a Type II
orbit in $I_1$.
Let us prove that $p=2$ and the statement of the proposition holds.
We distinguish three
cases.

\begin{itemize}
\item[a)]  If $\sigma \in  ]\sigma_{ent},\sigma_{smooth}[$,
then by Proposition \ref{propdisc} we have that $u(\xi_2^-) = u^+(\sigma) > u^*(\sigma)$.
As in Lemma \ref{luegotellamo}, it holds that $0 < u(\xi_2^+) < u^*(\sigma)$
and then, by Proposition \ref{viejo3} (see also Lemma \ref{viejo}),
$I_2 = ]\xi_2,+\infty[$ and uniqueness of (\ref{singular}) holds in
$I_2$. Thus,
the solution $u$ of (\ref{start}) in $I_2$ coincides with the solution
of (\ref{singular}) in that interval.

\item[b)]  If $\sigma = \sigma_{smooth}$, then  by
Proposition \ref{propdisc} we have that $u(\xi_2^-) = u^+(\sigma)= u^*(\sigma)$
and $r(\xi_2^-)=1$. By the Rankine--Hugoniot condition (\ref{sentcond})
and  observation
$(i)$ in Section \ref{lados}, respectively, we have $u(\xi_2^+) =
u^-(\sigma) = u^*(\sigma)$ and
$r(\xi_2^+)=1$. Again, by Proposition \ref{viejo3} (see also Lemma
\ref{viejo}),
$I_2 = ]\xi_2,+\infty[$ and uniqueness of (\ref{singular}) holds in
$I_2$. Thus,
the solution $u$ of (\ref{start}) in $I_2$ coincides with the solution
of (\ref{singular}) in that interval.
In this case, $u$ has an infinite slope at both sides of $\xi_2$ but $r$
matches continuously there.

\item[c)]  If $\sigma = \sigma_{ent}$, then  $u(\xi_2^-) = u^+(\sigma)>
u^*(\sigma)$
(see the proof of Lemma \ref{luegotellamo}). Recall that, by our
definition of $\sigma_{ent}$, we have
$u^-(\sigma_{ent}) = 0$. Since $u$ is an entropy solution, then
$r(\xi_2^-)=1$ and Rankine--Hugoniot condition
(\ref{sentcond}) holds. Finally,  we can ensure that
$u(\xi)= 0$ for $\xi\in ]\xi_2,\infty[$ as the set $\{0\}\times [-1,1] $ is a positively invariant manifold of the dynamical system. The solution coincides with the
traveling wave found in
Proposition \ref{propES}.
\end{itemize}

This concludes the proof. \end{proof}
\begin{remark}{\rm
We stress that as a consequence of Proposition \ref{unicidad}, no traveling structures (soliton-like solutions in particular) with range in $[0,1]$ and speed $\sigma \ge 0$ other than the ones given by Proposition \ref{propES} can be constructed. Regarding the case $\sigma <0$, a similar analysis could be carried to show that the only admissible traveling profiles in our framework are mirror images of those constructed for $\sigma>0$.}
\end{remark}
\begin{figure}[h]
\begin{center}
\includegraphics[width=9cm]{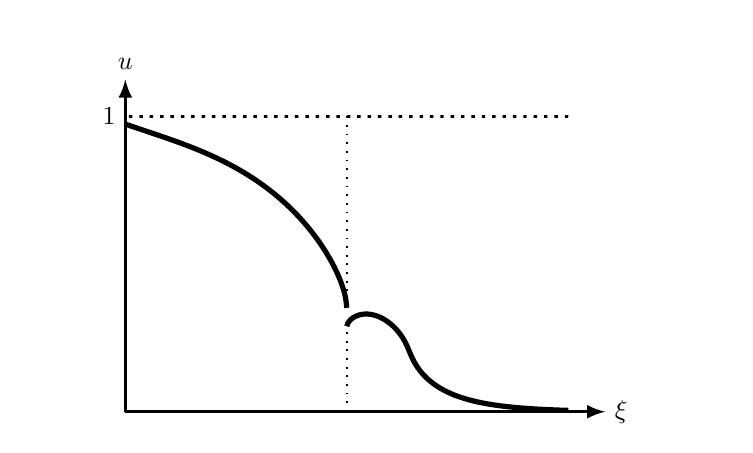}
\end{center}
\caption{An example of a piecewise smooth traveling wave profile which is not monotone that arises from a solution of \eqref{singular} and for which the wave speed  obeys \eqref{eq:speedp} for $\nu = c = 1$,  $m=2$, $F(u) = u(1-u)$ and $\sigma = 0,641233$. Vertical dotted lines show points with infinite slope. Note that the observation (i) at the beginning of Section \ref{lados}  
is not fulfilled, thus this profile does not solve \eqref{modelo1} in the entropic sense.}\label{soli}
\end{figure}

\begin{remark}
Let $u(t,x)$ be the entropy solution of (\ref{modelo1}) corresponding to an initial condition $u_0$
with compact support and such that $ \| u_0 \|_ {\infty} <  1$. By Proposition \ref{propES},
we know that there is an entropic  singular traveling wave, moving with velocity $\sigma_{ent}$, bounding $u_0$ from above.
As a consequence of the comparison principle in Theorem \ref{UniqSup},
there exists a positive constant $\beta$, depending on the support of $u_0$ and on the shape of the entropic traveling wave, such that
the support of  $u(t, \cdot)$ is contained in $(-\infty,\beta + \sigma_{ent} \, t)$.
Using a similar argument based on waves traveling to the left, the value $\beta$ can be chosen so that
the support of $u(t, \cdot)$ is contained in $(-\beta - \sigma_{ent} \, t,\beta + \sigma_{ent} \, t)$.
Thus $u(t,x)$ has compact support for any $t > 0$.

Note also that the traveling waves with support in a half line can be used to prove that solutions of
\begin{equation}\label{modelo1F0}
\begin{array}{ll}
\displaystyle
\frac{\partial u}{\partial t} = \nu \, \left(\frac{u^m u_x}{\sqrt{u^2 + \frac{\nu^2}{c^2}|u_x|^2}} \right)_x,\hspace{0.3cm} & {\rm in}
\hspace{0.2cm}  ]0,T[\times \R,
\end{array}
\end{equation}
corresponding to initial data with compact support are compactly supported.
Let us sketch the proof of this fact. Let $u_0 \in (L^1(\R)\cap L^\infty(\R))^+$ and assume that $u_0$
is supported in $[a,b]$. Let $u(t,x)$ be the entropy solution of
(\ref{modelo1F0}) with $u(0,x) = u_0(x)$. Observe first that the homogeneity (of degree $m>1$) of the operator in (\ref{modelo1F0})
implies that for any $\lambda > 0$, $u_\lambda(t,x)= \lambda^{1/(m-1)} u(\lambda t,x)$ is the entropy solution of
(\ref{modelo1F0}) with initial datum  $u_\lambda(0,x)= \lambda^{1/(m-1)} u_0(x)$. By an appropriate choice of $\lambda$ depending on $\Vert u_0\Vert_\infty$ and after a suitable translation of the initial profile
of $u_{\sigma_{{\rm ent}}}$ (eventually with $[a,b]$ in the interior of the support of $u_{\sigma_{{\rm ent}}}$)
we may ensure that $u_\lambda(0,x) \leq u_{\sigma_{{\rm ent}}}(x)$, $x\in\R$.
Since $u_{\sigma_{{\rm ent}}}(x-\sigma_{{\rm ent}} t)$ is a super-solution of (\ref{modelo1F0}),
by the comparison principle in Theorem \ref{UniqSup}.(ii) we have that
$u_\lambda(t,x) \leq u_{\sigma_{{\rm ent}}}(x-\sigma_{{\rm ent}} t)$ for any $t > 0$. Writing this inequality in terms of $u(t,x)$ we have
$u(t,x) \leq \lambda^{-1/(m-1)} u_{\sigma_{{\rm ent}}}(x-\sigma_{{\rm ent}} \frac{t}{\lambda})$. By comparing with a traveling wave
moving to the left with speed $\sigma_{{\rm ent}}$ we deduce that for any $t > 0$ the support of $u(t)$ is contained in
$\left[ a -\epsilon - \sigma_{{\rm ent}} \frac{t}{\lambda}, b + \epsilon+ \sigma_{{\rm ent}} \frac{t}{\lambda}\right]$ for some
$\epsilon,\lambda > 0$ determined by $u_0$.
\end{remark}

\subsection{Numerical insights about traveling waves viewed as attractors}\label{SectNumerics}
Studying the stability of traveling wave solutions and their dynamic ability to attract other solutions  is a very interesting problem that is beyond the scope of this paper.
Another  problem that will surely open new lines of research is to understand how the saturation of diffusion produces shocks (and the role played by the
reaction terms, if any).
 The idea of this  paragraph is to give some insights of how these two problems raise new challenges in this context. To do that we use the numerical solutions of the dynamical system associated with traveling waves \eqref{singular} together with those  associated to  the partial differential equation \eqref{modelo1}. In Figure \ref{cerca} we have represented both curves associated to different settings. The traveling wave profiles have been displaced by matching its discontinuities with those of the  time dependent solutions.
For the numerical solution of the time-dependent problem we have used a WENO  solver   together with a Runge--Kutta scheme.

Figure  \ref{cerca}A) describes the evolution of an initial data with compact support and how  does it evolve  into an  entropic jump which locally around the front  behaves like a traveling wave.
By using the comparison principle for solutions in Theorem \ref{UniqSup}, we deduce that the traveling wave will be above the time dependent solution of the system \eqref{modelo1}.
The entropic traveling wave  provides an upper estimate of the growth rate of the support. Let us precise in the following result the order of the singularity of the traveling wave solution near the jump $\xi_2$, where we use the notation of Step 4 in Proposition \ref{unicidad}.

\begin{lemma} \label{converge}
Let $u$ be an entropic traveling wave for $\sigma \in [\sigma _{ent}, \sigma_{smooth}[$. Then,  the vertical angle near the jump  $\xi_2$ is of order $|\xi -\xi_2|^{-\frac{1}{3}}$.
\end{lemma}
\begin{proof}
Using  (\ref{u-grafo}), we obtain that the points $u_{\pm }$
at which $r(u)$ touches the edge  $r=1$   verify
\[
\lim_{u\rightarrow u_{\pm }}\left| \frac{r(u)-1}{ u-u_{\pm} }%
\right| =\lim_{u\rightarrow u_{\pm }}\left| r^{\prime }(u)\right| =\alpha ,
\]
where $\alpha =\frac 1{cu_{\pm }}\left| mc-\sigma u_{\pm
}^{1-m}\right|$.  Hence, combining  (\ref{r-sigma}) together with the approximation  $\sqrt{1-r^2}%
\sim \sqrt{1-r}\sqrt{2}$ we deduce for  $\sigma_{ent}<
\sigma<\sigma_{smooth}$, which implies $u_{\pm }\neq u^{*}$, the following equality
\[
\lim_{\xi \rightarrow \xi _2^{\pm }}\left| u(\xi )-u_{_{\mp }}\right|
^{\frac 12} |u^{\prime }(\xi )|=\frac{cu_{\mp }}{\sqrt{2 \alpha} \, \nu  }=\frac{%
(c\, u_{\mp })^{\frac3 2}}{\sqrt{2}\, \nu \left| mc-\sigma u_{\mp }^{1-m}\right|^{\frac1 2}}.
\]
Then, we have
\[
\lim_{\xi \rightarrow \xi _2^{\pm }}\frac{\frac 23\left| u(\xi )-u_{_{\mp
}}\right| ^{\frac 32}}{\left| \xi -\xi _2\right| }=\frac{%
(c\, u_{\mp })^{\frac3 2}}{\sqrt{2}\, \nu \left| mc-\sigma u_{\mp }^{1-m}\right|^{\frac1 2} }
\]
or equivalently
\[
\lim_{\xi \rightarrow \xi _2^{\pm }}\frac{\left| u(\xi )-u_{\mp }\right| }{%
\left| \xi -\xi _2\right| ^{\frac 23}}=\left( \frac{%
3 (c\, u_{\mp })^{\frac3 2}}{2 \sqrt{2}\, \nu \left| mc-\sigma u_{\mp }^{1-m}\right|^{\frac1 2} }\right) ^{\frac 23}.
\]
Letting  $\beta _{\pm }=\left( \frac{%
3 (c\, u_{\pm })^{\frac3 2}}{2 \sqrt{2}\, \nu \left| mc-\sigma u_{\pm }^{1-m}\right|^{\frac1 2} }\right) ^{\frac 23}$,  we find
\[
\begin{array}{cc}
u(\xi )\sim u_{+}+\beta _{+}\left| \xi -\xi _2\right| ^{\frac 23}, & \xi <\xi
_2 ,\\
u(\xi )\sim u_{-}-\beta _{-}\left| \xi -\xi _2\right| ^{\frac 23}, & \xi >\xi
_2.
\end{array}
\]
For  $\sigma =\sigma _{ent}$, since  $u(\xi )=0=u_-=0$ for $\xi
>\xi _2$,  taking into account that $\sigma= c u_+^{m-1}$,  we analogously obtain 
\begin{eqnarray}\label{powerlaw}
\qquad u(\xi )\sim  \left(\frac{\sigma}{c}\right)^{\frac{1}{m-1}}+ \frac1 2 \left(\frac{3c}{\nu  \left( m-1\right)^{\frac1 2}}\right)^{\frac 23}  
 \left(\frac{\sigma}{c}\right)^{\frac{1}{m-1}}  \left| \xi -\xi _2\right| ^{\frac 23}, \qquad  \xi <\xi _2,
\end{eqnarray}
which provides an estimate of the order of approximation of the  traveling wave profile near the front. \end{proof}

The numerical time dependent solutions given in  Figure  \ref{cerca}A) has the same power law behavior near the front than the corresponding traveling wave 
\eqref{powerlaw}.

In Figure  \ref{cerca}B)
the numerical calculations show spontaneous singularization of solutions and the convergence of an initial data towards a traveling wave solution to the type described in Figure \ref{TWprofiles}C).

\begin{figure}[h]
\begin{center}
\includegraphics[width=13cm]{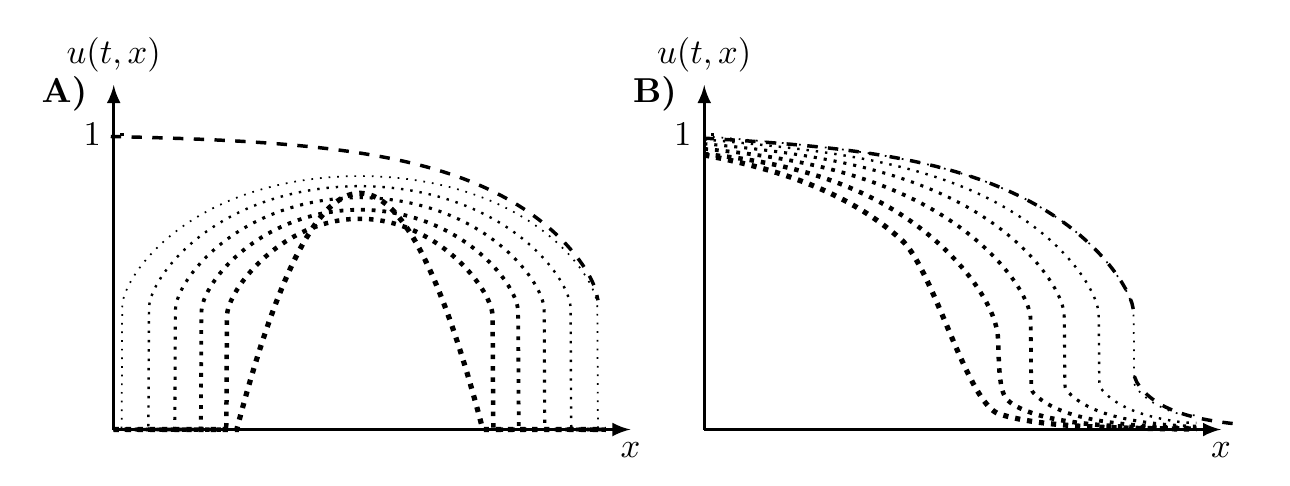}
\end{center}
\caption{Dotted lines represent time dependent solutions for  \eqref{modelo1} (the smaller the dots, the more advanced the times%
).
Dashed lines describe traveling wave profiles obtained from \eqref{singular} for  different { values of $\sigma$}.
In all cases we have used $\nu = c = 1$,  $F(u) = u(1-u)$ and $m=2$.
{A)} Time evolution of a compactly supported initial condition with null Dirichlet conditions. The traveling wave profile depicted, corresponding to
$\sigma_{ent} = 0,437803$, constitutes a super-solution for the time-dependent solution.
{B)} Time evolution of a regular initial condition, with Neumann boundary conditions. The traveling wave profile depicted corresponds to
$\sigma= 0,57 \in ]\sigma_{ent},\sigma_{smooth}[$.
}\label{cerca}
\end{figure}

\subsection{The $L^p$-continuity w.r.t. the wave speed}

The purpose of this paragraph is to prove the continuity of the traveling profiles constructed in Proposition \ref{propES} with respect to the wave velocity. In order to do that it is convenient to choose a privileged normalization for the traveling profiles, so that we get a family $u^{\mathcal{N}}(\sigma)$ defined in a unique way. We do this as follows:

\begin{definition}
Let $u_\sigma^{\mathcal{N}}$  with  $\sigma \in [\sigma_{ent},+\infty[$ be the family of traveling wave solutions constructed in Proposition \ref{propES}
and enjoying the following additional properties:
\begin{itemize}

\item If $\sigma > \sigma_{smooth}$ we set $u_\sigma^{\mathcal{N}}(0)=u^*(\sigma_{smooth})=u^+(\sigma_{smooth})=u^-(\sigma_{smooth})$,

\item If $\sigma_{smooth}\ge \sigma > \sigma_{ent}$ we set $ \lim_{\xi \to 0^\mp} u_\sigma^{\mathcal{N}}(\xi)=u^\pm(\sigma)$.

\item If $\sigma = \sigma_{ent}$ we set $\lim_{\xi \uparrow 0}u_\sigma^{\mathcal{N}}(\xi)=u^+(\sigma)$ and $u_\sigma^{\mathcal{N}}(\xi)=0,$ for $\xi>0$.

\end{itemize}
\end{definition}

We assume in this Section that $F$ satisfies
\begin{equation}\label{arap}
\hbox{\rm there is some $p\geq 1$ such that $\liminf_{u \to 0}\frac{F(u)}{u^p} = k \in]0, +\infty]$.}
\end{equation}
Now we are going to prove that this specially parameterized family of traveling wave solutions has the  property of continuous dependence with respect to the wave speed stated in Theorem \ref{parto}, that is,
\begin{equation}\label{arap2}
\lim_{\sigma_1 \to \sigma_2} \|u^{\mathcal{N}}_{\sigma_1}(t)-u^{\mathcal{N}}_{\sigma_2}(t)Ê\|_{L^p(\R)} +   \|u^{\mathcal{N}}_{\sigma_1}(t)-u^{\mathcal{N}}_{\sigma_2}(t)Ê\|_{L^\infty (\R)}= 0,
\end{equation}
for any $t\geq 0$ and any $\sigma_1, \sigma_2 \in [\sigma_{ent},+\infty[$.

\begin{figure}[h]
\begin{center}
\includegraphics[width=9cm]{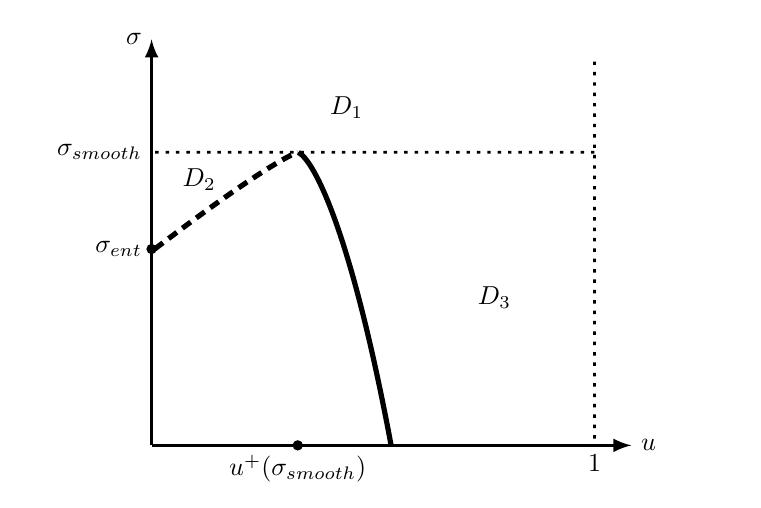}
\end{center}
\caption{Numerical representation of the set $D$ for $\nu = c = 1$,  $m=2$, $F(u) = u(1-u)$. Continuous and dashed lines represent, respectively, the inverse of the mappings $u^+(\sigma)$ and $u^-(\sigma)$.   }\label{bifurcation}
\end{figure}
To prove the previous result we will need to argue with pairs of the form $(u,\sigma)\in ]0,1[\times]0,+\infty[$. We consider the following subsets:
\begin{eqnarray*}
D_1&=&]0,1[\times]\sigma_{smooth},+\infty[,
\\
D_2&=&\{(u,\sigma)\in ]0,1[\times]0,+\infty[ : \sigma_{ent}< \sigma \le \sigma_{smooth},\ 0<u< u^-(\sigma)\},
\\
D_3&=&\{(u,\sigma)\in ]0,1[\times]0,+\infty[ : 0< \sigma \le \sigma_{smooth},\ u^+(\sigma)< u<1\},
\end{eqnarray*}
and finally we let $D=D_1\cup D_2\cup D_3$.

We observe that the map $(u,\sigma) \mapsto r_\sigma (u)$ is defined at least on $D_1\cup D_3$. To extend it to $D_2$ we choose
$r_\sigma (u)$ as the solution of (\ref{u-grafo}) defined over $]0,u^- (\sigma)[$ and such that $r_\sigma (u^-(\sigma))=1$.
This extension is justified by Lemma \ref{viejo}.

\begin{lemma}\label{proprsigma}
The function $(u,\sigma) \in D \mapsto r_\sigma (u)$ is continuous. Moreover, it has the following properties:
\begin{enumerate}
\item[(i)] Let $\sigma_{ent} < \sigma_1 < \sigma_2 \leq \sigma_{smooth}$, then $r_{\sigma_1}(u)> r_{\sigma_2}(u)$, for any $u \in ]0, u^-(\sigma_1)[$  (which is their common interval of definition).
\item[(ii)] Let $\sigma_n \to \sigma_0 \geq 0$, then there  exist  $u_0 \in ]0,1[$ and $L>0$ such that
 \begin{equation}
 \label{2*}
 \displaystyle \frac{(1-u) \sqrt{1 - r^2_{\sigma_n}(u)}}{u \, r_{\sigma_n}(u)} \geq L, \quad \forall u \in]u_0, 1[.
\end{equation}
 Let $ \sigma_n,\sigma_0 > \sigma_{ent}$ and $\sigma_n \to \sigma_{ent}$. Then, there exist $u_0 \in ]0,1[$ and $L>0$ such that
 \begin{equation}
 \label{3*}
 \displaystyle \frac{ \sqrt{1 - r^2_{\sigma_n}(u)}}{ r_{\sigma_n}(u)} \geq L, \quad \forall u \in]0, u_0[.
\end{equation}
\end{enumerate}
\end{lemma}

\begin{proof}
Note that the continuity in  $D_1\cup D_3$ follows from the uniform convergence on compact sets given in Lemma \ref{cont}.
On the other hand, the continuity on $D_2$ follows from Peano's theorem on continuous dependence with respect to initial conditions, which we apply to any continuous extension of (\ref{u-grafo}). Thus, our claim will be proved if we are able to ensure that the values of $r_\sigma (u)$ in $]0,u^- (\sigma_{smooth})[ \times \{ \sigma_{smooth} \}$ obtained by its extension to $D_2$ coincide with the values  obtained by performing limits of points in $D_1$.

We pick sequences $\sigma_n \downarrow \sigma_{smooth}$ and $u_n \to u_0\in ]0,u^- (\sigma_{smooth})[ $. Coming back to Lemma \ref{viejo}, after  bounding $r'_{\sigma_n}$ and using Ascoli's theorem, we notice that the limit of $r_{\sigma_n}$ in $]0,u^- (\sigma_{smooth})[ $ (recall that $r_{\sigma_n} (\cdot)$ is defined over $]0,1[$) gives a Type III orbit such that $ r (u^+(\sigma_{smooth})=u^*(\sigma_{smooth}))=1$. Noting that the convergence is uniform over compact sets, this allows to conclude.

To prove   assertion \emph{(i)}, we argue by contradiction. For that we consider that there exists $\tilde u \in ]0,u^-(\sigma_1)[$ such that
\begin{equation}\label{kk33}
r_{\sigma_1}(\tilde u)\leq r_{\sigma_2}(\tilde u) .
\end{equation}
As $r_{\sigma_2}( u^-(\sigma_1))<1=r_{\sigma_1}( u^-(\sigma_1))$,
we can assume that $\tilde u$ is the last point for which \eqref{kk33} is verified. Then $r_{\sigma_1}(\tilde u)= r_{\sigma_2}(\tilde u)$ and $r'_{\sigma_2}(\tilde u)\leq r'_{\sigma_1}(\tilde u)$. By equation \eqref{u-grafo}, this contradicts the fact that $\sigma_1 < \sigma_2$.

Next we move to assertion \emph{(ii)}; we consider the upper bound \eqref{2*} in first place. Let $u_0 \in ]0,1[$ such that
the left hand side of \eqref{2*} is defined in $[u_0, 1[$, for any $n$. To do that, let $\bar \sigma = \inf_{n \in\NN} \sigma_n$ and we choose $u_0 \in ]u^+(\bar \sigma), 1[$, or $u_0 \in ]0, 1[$ if $\bar \sigma > \sigma_{smooth}$.
Now we use the monotonicity of $r_{\sigma}$ with respect to $\sigma$ together with the monotonicity of the function $r \mapsto \frac{\sqrt{1 - r^2}}{r}$ to devise the following estimate
\begin{equation}
\label{4*}
\frac{(1-u) \sqrt{1 - r^2_{\sigma_n}(u)}}{u \, r_{\sigma_n}(u)} \geq  \frac{1-u}{u} \frac{\sqrt{1 - r_{\bar \sigma}^2(u)}}{r_{\bar \sigma}(u)}, \qquad \forall u\in ]u_0, 1[.
\end{equation}
The right hand side of the \eqref{4*} is positive and bounded from below by a constant $L>0$, since it is bounded at both  ends of the interval $]u_0, 1[$ (note that $\lim_{u \to 1} \frac{1-u}{u} \frac{\sqrt{1 - r_{\bar \sigma}^2(u)}}{r_{\bar \sigma}(u)} = - \frac{1}{r^\prime_{\bar \sigma}(1)} > 0$).

To deal with the second estimate \eqref{3*}, let $u_0 \in ]0,1[$ in such a way that the left hand side of \eqref{3*} is well-defined in $]0,u_0[$. To do that, we consider again $\bar \sigma = \inf_{n \in\NN} \sigma_n$ and  $u_0 \in ]0, u^-(\bar \sigma)[$, or $u_0 \in ]0, 1[$ if $\bar \sigma > \sigma_{smooth}$, which is consistent with the fact that $\bar \sigma > \sigma_{ent}$. Then, as in the previous case we can estimate from below
\begin{equation}
\label{5*}
\frac{\sqrt{1 - r^2_{\sigma_n}(u)}}{r_{\sigma_n}(u)}
\geq   \frac{\sqrt{1 - r_{\bar \sigma}^2(u)}}{r_{\bar \sigma}(u)}, \qquad \forall u\in ]0, u_0[.
\end{equation}
Now, taking into account
$$
\lim_{u \to 0} \displaystyle  \frac{\sqrt{1 - r_{\bar \sigma}^2(u)}}{r_{\bar \sigma}(u)} = \left\{
\begin{array}{ll}
\displaystyle  \frac{\sqrt{1 - (r^*)^2}}{r^*} , & \mbox{ if } r^*>0, \\
 \\
 +\infty, & \mbox{ if } r^* =0,
\end{array}
\right.
$$
we can deduce that the  right hand side of the \eqref{5*} is positive and bounded from below by a constant $L>0$, which concludes the proof.
\end{proof}

Making use of the graphs $u \mapsto r_\sigma (u)$ we are able to introduce the following function:
$$
G(u,\sigma) = \left\{
\begin{array}{ll}
- \displaystyle \int_{u^*(\sigma_{smooth})}^u \frac{\sqrt{1-r_\sigma^2(v)}}{v r_\sigma (v)}dv,&  \mbox{in}\ D_1,
\\  \\
- \displaystyle \int_{u^-(\sigma)}^u \frac{\sqrt{1-r_\sigma^2(v)}}{v r_\sigma (v)}dv, & \mbox{in}\ D_2,
 \\  \\
 - \displaystyle \int_{u^+(\sigma)}^u \frac{\sqrt{1-r_\sigma^2(v)}}{v r_\sigma (v)}dv,
& \mbox{in}\ D_3.
\end{array}
\right.
$$
We can use this function to recover the traveling wave profiles $u_\sigma (\xi)$.
\begin{lemma}
\label{reprform}
 For any $\sigma >\sigma_{ent}$  and $\xi \neq 0$ we have that
$$
\frac{c}{\nu}G(u_\sigma^{\mathcal{N}}(\xi),\sigma)=\xi.
$$
\end{lemma}
\begin{proof}
We argue first for $\xi> 0$. Choose $0<\xi_1<\xi$, and integrate (\ref{r-sigma}) between $\xi_1$ and $\xi$ to get
$$
-\frac{c}{\nu}\int_{\xi_1}^{\xi}
 \frac{\sqrt{1-r_\sigma^2(u_{\sigma}^{\mathcal{N}}(\eta))}}{u_{\sigma}^{\mathcal{N}}(\eta) r_\sigma (u_{\sigma}^{\mathcal{N}}(\eta))}
(u_{\sigma}^{\mathcal{N}})'(\eta) d \eta =\xi -\xi_1.
$$
Now, after the change of variables $v=u_{\sigma}^{\mathcal{N}}(\eta)$, we arrive to
$$
-\frac{c}{\nu}\int_{u_{\sigma}^{\mathcal{N}}(\xi_1)}^{u_{\sigma}^{\mathcal{N}}(\xi)}
 \frac{\sqrt{1-r_\sigma^2(v)}}{v r_\sigma (v)}dv=\xi -\xi_1.
$$
Finally, we let $\xi_1\to 0$; observe that the integrand is positive and in particular the integral exists. We can argue in a similar way if  $\xi< 0$.
\end{proof}
\begin{remark}
Note that for $\sigma_{ent}\leq \sigma < \sigma_{smooth}$ the traveling wave solutions $u_\sigma^{\mathcal{N}}$ are not well
defined for $\xi=0$.
\end{remark}

Let us focus now on the properties of  $G$.
\begin{proposition}\label{propG}
The following properties are satisfied:
\begin{enumerate}
\item The function $G$ is continuous over $D$.
\item The following assertions give the behavior of $G$ at the boundary of $D$:
\begin{itemize}
\item[  i)] $G$ tends to zero when we approach any point of the set
$$
\{(u^+(\sigma),\sigma) : 0<\sigma\leq \sigma_{smooth} \}\cup \{(u^-(\sigma),\sigma) : \sigma_{ent}<\sigma\leq \sigma_{smooth}\}\,,
$$
\item[ ii)] $G$ tends to $-\infty$ when we approach any point of the set  $\{1\} \times  [0,\infty[$\,,
\item[iii)] $G$ tends to $+\infty$ when we approach any point of the set $\{0\} \times ] \sigma_{ent},\infty[$.
\end{itemize}
\end{enumerate}
\end{proposition}
\begin{proof}
Since the function $(u,\sigma) \in D \mapsto r_\sigma (u)$ is continuous (Lemma \ref{proprsigma}), then $r_{\sigma}$ can be extended continuously by $1$ to  $]0,1[\times]0,\infty[ \setminus D$.
Thus, $G$ is also continuous because it is given by the integral of a continuous function depending continuously on $\sigma$, the integral being extended to intervals
that also depend continuously on $(u,\sigma)$.
This proves  assertions {\it (1).} and {\it (2).i).}

To prove {\it (2). ii)}, let $\sigma_n \to \sigma_0 $ and $u_n\to 1$. By Lemma \ref{proprsigma}.$(ii)$
$$ \frac{\sqrt{1-r^{2}_{\sigma_n}(u)}}{u r_{\sigma_n} (u)}\ge \frac{L}{1-u}$$
is satisfied on  some interval $]u_0,1[$. This leads us to
\begin{equation*}
 \label{unos}
G(u,\sigma_n)\leq h+L \ln (1-u),
\end{equation*}
which holds for $u_0<u<1$, where $h$, $L$ are positive constants not depending on $n$. Then, {\it (2). ii)} follows.

Let now $\sigma_n \to \sigma_0 > \sigma_{ent}$ and $u_n \to 0$. Using Lemma \ref{proprsigma}.$(ii)$ again we find an interval $]0,u_0[$ for which
 $$ \frac{\sqrt{1-r^{2}_{\sigma_n}(u)}}{u r_{\sigma_n} (u)}\ge \frac{L}{u}.$$
After integration in $[u^*(\sigma_{smooth}),u]$ if $\sigma_0 \geq \sigma_{smooth}$ or in
$[u^-(\sigma),u]$ if $\sigma_0 < \sigma_{smooth}$ (both intervals coincide if $\sigma_0 = \sigma_{smooth}$)
we obtain
\begin{equation*}
\label{doss}
 G(u,\sigma_n) \geq h - L \ln (u),
\end{equation*}
for some positive constants $h,\ L$  which do not depend on $n$.  This proves {\it (2). iii)}.

\end{proof}

\begin{lemma}
\label{AUX}
The map $\sigma \to  u_\sigma^{\mathcal{N}}(\xi) $ is monotonically decreasing for $\xi >0$ and monotonically increasing for $\xi<0$.
\end{lemma}
\begin{proof}
Thanks to Lemma \ref{proprsigma}. \emph{(i)} we have that the mapping $\sigma \mapsto \frac{\sqrt{1-r_\sigma^2(v)}}{r_\sigma(v)}$ is monotonically increasing for any fixed $v\in [0,1]$. Next, we note that $u<u^-$ in $D_2$, with $\sigma \mapsto u^-$ increasing. We also have that $u>u^+$ in $D_3$, with $\sigma \mapsto u^+$ decreasing. We combine the previous information with the representation formula for $G$  given by Lemma \ref{reprform} to obtain the result.
\end{proof}

\begin{proposition}\label{propUN}
Let $\{\sigma_n \}\subset [\sigma_{ent},+\infty[$ and $\sigma_n \to \sigma_0>\sigma_{ent}$. Then
 for any $T>0$ the sequence $u_{\sigma_n}^{\mathcal{N}}$ converges uniformly on $[-T,T] \backslash \{0\}$.
\end{proposition}
\begin{proof}
To check the uniform convergence of $u_{\sigma_n}^{\mathcal{N}}$ on $]0,T]$ we argue on $[0,T]$, after
extending the functions $u^{\mathcal{N}}_{\sigma_n}$ to $\xi=0$ using either $u^*(\sigma_{smooth})$, in case that
$\sigma_0 \geq \sigma_{smooth}$, or $u^-(\sigma_n)$, in case that $\sigma_0 \in ]\sigma_{ent},\sigma_{smooth}[$.

We use the characterization of the uniform convergence by sequences:   given any fixed sequence $\xi_n\geq 0$ that converges to some $\xi_0> 0$, we are to show that $u^{\mathcal{N}}_{\sigma_n}(\xi_n) - u_{\sigma_0}(\xi_n)$ converges to zero. Since
$$
\frac{c}{\nu}G(u^{\mathcal{N}}_{\sigma_n}(\xi_n),\sigma_n)=\xi_n
$$
is bounded, then  $u^{\mathcal{N}}_{\sigma_n}(\xi_n)$ stays in $]0,1[$ thanks to Proposition \ref{propG}. Any convergent subsequences of  $u^{\mathcal{N}}_{\sigma_n}(\xi_n)$ will converge to a point
$u_0\in ]0,1[$, which may in principle depend on the subsequence. Taking the limit along any such subsequence we get
\begin{equation}\label{kk22}
\frac{c}{\nu}G(u_0,\sigma_0)=\xi_0.
\end{equation}
This relation is solved only 
by $u_0=u_{\sigma_0}^{\mathcal{N}}(\xi_0)$, no matter if we are on $0<u \leq u^-(\sigma_0)$ or on $0<u \leq u^*(\sigma_{smooth})$ -- note that $G$ does not change sign within $D_2$ nor in $D_3,\ D_1 \cap \{u>u^*\}$ or $D_1 \cap \{u<u^*\}$. This shows that in fact $u_{\sigma_n}^{\mathcal{N}}(\xi_n)\to u_{\sigma_0}^{\mathcal{N}}(\xi_0)$ for the whole sequence, and our claim follows.

The case $\xi_n \to \xi_0= 0$ requires a more detailed analysis because in this case \eqref{kk22} may have two solutions: $u^+(\sigma_0)$ and $u^-(\sigma_0)$.  However, in this case $u^{\mathcal{N}}_{\sigma_n}(\xi_n) \leq \lim_{\xi \downarrow 0} u^{\mathcal{N}}_{\sigma_n}(\xi) = u^-(\sigma_n)$ and, as a consequence, the corresponding limit  $\bar u_0$ verifies $\bar u_0 \leq u^-(\sigma_0)$. Using \eqref{kk22} for $\xi_0 =0$, we deduce that $\bar u_0 = u^-(\sigma_0)$, which coincides with the extension we made at the beginning of this proof.

The proof of the uniform convergence over $[-T,0[$ is similar and we omit the details.
\end{proof}

\begin{remark} \label{ult}
 The above result is still valid in the case $\sigma_n \to \sigma_0=\sigma_{ent}$ in the interval $]-T, 0[$, for any $T>0$.
\end{remark}

Let us now prove the uniform continuity of the traveling wave profiles with respect to $\sigma$.
Consider a sequence $\sigma_n \to \sigma_0 \geq \sigma_{ent}$. In a first step we study the case $ \sigma_0 > \sigma_{ent}$.
By Proposition \ref{propUN}, it is enough to prove that $u^{\mathcal{N}}_{\sigma_n}(\xi_n) \to 0$
as $\xi_n \to +\infty$, and $u^{\mathcal{N}}_{\sigma_n}(\xi_n) \to 1$ as $\xi_n \to -\infty$, since this obviously  implies that $\left| u^{\mathcal{N}}_{\sigma_n}(\xi_n) - u^{\mathcal{N}}_{\sigma_0}(\xi_n)\right| \to 0$. Being both assertions similar, let us prove only the first one.
For that, we note that given $\epsilon >0$, there exists $\bar \xi$ such that $u^{\mathcal{N}}_{\sigma_0}(\bar \xi) < \epsilon$. Since, by Proposition \ref{propUN},
$u^{\mathcal{N}}_{\sigma_n}(\bar \xi) \to u^{\mathcal{N}}_{\sigma_0}(\bar \xi)$, there is a value $n_0$ such that $u^{\mathcal{N}}_{\sigma_n}(\bar \xi) < \epsilon$ for $n>n_0$. Let $n_1 \in \NN$ be such that $\xi_n \geq \bar \xi$ for any $n> n_1$. Then, choosing $n > \max \{n_0, n_1\}$, we find that $u^{\mathcal{N}}_{\sigma_n}(\xi_n) \leq u^{\mathcal{N}}_{\sigma_n}(\bar \xi) < \epsilon$, thanks to Lemma \ref{AUX}.

In case that $\sigma_n \to  \sigma_{ent}$ we have to distinguish between $]-\infty, 0[$ and $]0, \infty[$. If $\xi \in ]-\infty, 0[$ we  use Remark \ref{ult} and 
Lemma \ref{AUX} to deduce the same result. For $\xi \in ]0, \infty[$ we conclude by using the bound $u^{\mathcal{N}}_{\sigma_n}(\xi) \leq u^-(\sigma_n)$.

Finally we end the proof of (\ref{arap2}) by
proving the continuity of the traveling waves with respect to $\sigma$ in $L^p(\R)$ , where $p$ is given by (\ref{arap}).
This will conclude the proof of the continuity assertions of Theorem \ref{parto}.

First, let us prove that
\begin{equation}\label{intpi}
u^{\mathcal{N}}_{\sigma} \in L^p(\R^+) \quad \hbox{\rm and} \quad 1- u^{\mathcal{N}}_{\sigma} \in L^1(\R^-).
\end{equation}
For that we notice that (\ref{intpi}) holds if
$F(u^{\mathcal{N}}_\sigma )\in L^1(\R)$. This is a consequence of the fact
$$
\lim_{\xi \to \infty} \frac{F\big(u^{\mathcal{N}}_\sigma(\xi)\big)}{\big(u^{\mathcal{N}}_{\sigma}(\xi)\big)^p } = \liminf_{u \to 0} \frac{F(u)}{\left(u\right)^p} =k \in]0,+\infty] \, ,
$$
$$
\lim_{\xi \to -\infty} \frac{F\big(u^{\mathcal{N}}_\sigma(\xi)\big)}{1- u^{\mathcal{N}}_{\sigma}(\xi) } = \lim_{u \to 1} \frac{F(u)}{1-u} =-F'(1) > 0  \, .
$$
Now we prove the integrability of  $F(u^{\mathcal{N}}_\sigma )$ over the whole real line.
For that we rewrite (\ref{start}) as
$$
F(u(\xi)) = \nu \left(c u^m(\xi) r(\xi)- \sigma u(\xi) \right)'.
$$
Integrating the previous relation and using the boundedness of $F$, and the finiteness of $\lim_{\pm \infty} u$ and $\lim_{\pm \infty} r$, we get
that $F(u^{\mathcal{N}}_\sigma) \in L^1(\R)$. Hence, (\ref{intpi}) holds.

Since $1- u^{\mathcal{N}}_{\sigma} \in L^1(\R^-)$ and $|1- u^{\mathcal{N}}_{\sigma}|\leq 1$, we also have that
\begin{equation}\label{intpi2}
1- u^{\mathcal{N}}_{\sigma}\in L^p(\R^-).
\end{equation}
This allows us to conclude the convergence of any sequence $u^{\mathcal{N}}_{\sigma_n}$ to $u^{\mathcal{N}}_{\sigma_0}$ in $L^p(\R)$
as $\sigma_n \to \sigma_0 \geq \sigma_{ent}$. Indeed, by Proposition \ref{propUN} and (\ref{intpi2}) the sequence
$\left| u^{\mathcal{N}}_{\sigma_n} - u^{\mathcal{N}}_{\sigma_0}\right|^p$ is dominated by a function in $L^p(\R)$
and converges pointwise to $0$. The result follows as a consequence of the Dominated Convergence Theorem.

 \begin{remark}
  Note that in the proof of the uniform convergence of $u_{\sigma_n}^{\mathcal{N}}$ we have not used any hypothesis on the asymptotic behavior of $F$ at $0$. Note also that under the hypothesis $K(0) >0$ the  $L^1(\R)$ convergence holds, since this hypothesis implies that \eqref{arap} is fulfilled for $p=1$.
\end{remark}

\section{Appendix: Entropy solutions}\label{preliminaris}

Our purpose in this Appendix is to give the necessary background in order to introduce
the notion of entropy solutions to (\ref{modelo1}),
to state some existence and uniqueness results for them, and to
give sense to the properties stated in Section \ref{sect:analysisEC}.

Equation (\ref{modelo1}) belongs to the more general class of flux limited  diffusion equations, which has been extensively studied in
\cite{ACMEllipticFLDE,ACMMRelat,ACMSV,CEU2,leysalto}. As shown in those papers, the
notion of entropy solution is the right one in order to
prove existence and uniqueness results and to describe the qualitative features of solutions.
In particular, and closely related to this work, the so-called relativistic heat equation (which corresponds to $m=1$ in (\ref{modelo1}))
coupled with a Fisher--Kolmogorov type reaction term has been studied in
\cite{limitedFKPP,CGSS}. Existence and uniqueness results for that model were proved in \cite{limitedFKPP}, the construction of traveling waves
being the object of \cite{CGSS}.

Thus, our first purpose is to give a brief review of the concept of entropy solution for flux limited diffusion equations.
Although we are only concerned with the case $N=1$, we state the results in the more general context where $N\geq 1$ since this may be useful for future reference.
For a more detailed treatment we refer to  \cite{ACMMRelat,CEU2}.
We consider parabolic equations of the form
\begin{equation} \label{DirichletproblemP} \left\{
\begin{array}{ll}
\displaystyle \frac{\partial u}{\partial t} = \div\, \a(u,
\nabla u)+F(u),
 &
\hspace{0.3cm}\hbox{in \hspace{0.2cm} $Q_T=]0,T[\times \R^N$}\\
\displaystyle
 \\
\displaystyle u(0,x) = u_{0}(x), & \hspace{0.3cm} \hbox{in
\hspace{0.2cm} $x \in \R^N$}
\end{array}
\right.
\end{equation}
where $F(u)$ is a Lipschitz continuous function such that $F(0)=0$ and $\a(z, \zeta) = \nabla_{\zeta}f(z, \zeta)$  is
associated to a Lagrangian $f$ satisfying a set of technical assumptions. Let us give a brief account of them, referring to
\cite{ACMMRelat,CEU2} for a thorough presentation. Thus, we assume that

\begin{quote} \noindent (H) $f$ is continuous on $[0,\infty[ \times \R^N$ and is a
convex differentiable function of $\zeta$ such that
$\nabla_{\zeta}f(z, \zeta) \in C([0,\infty[ \times \R^N)$. Further, we require
$f$ to satisfy the coercivity and linear growth conditions
\begin{equation}\label{linearGr}
C_0(z) \vert \zeta \vert - D_0(z) \leq f(z, \zeta) \leq M_0(z)(\vert \zeta \vert + 1),
\end{equation}
for any $(z,\zeta)\in [0,\infty[\times \R^N$,
and some positive and continuous functions $C_0, D_0,$ $ M_0 \in C([0,\infty[)$ with $C_0(z) > 0$
for any $z\neq 0$. Notice  that
$\vert \zeta \vert$ denotes the Euclidian norm of $\zeta\in\R^N$. We assume that
\begin{equation*}
\nonumber
C_0(z) \geq c_0 z^{m}, \quad \hbox{\rm for some $c_0 > 0$, ${m} \geq 1$, $z\in [0,\infty[$.}
\end{equation*}
\end{quote}
Let $\a(z, \zeta) = \nabla_{\zeta}f(z, \zeta)$, $(z,\zeta)\in [0,\infty[\times \R^N$. We assume that
there is a vector field $\b(z,\zeta)$ and a constant $M > 0$ such that
\begin{equation}\label{acotacionbb}
\a(z,\zeta) = z^{m} \b(z,\zeta) \quad \hbox{\rm with}\quad  \vert \b (z, \zeta) \vert \leq M, \ \ \ \forall \ (z, \zeta) \in
[0,\infty[ \times \R^N.
\end{equation}

We consider the function $h : [0,\infty[ \times \R^N \rightarrow \R$
defined by
\begin{equation}\label{defh}
h(z, \zeta):= \a(z, \zeta) \cdot \zeta.
\end{equation}
From the convexity of $f$ in $\zeta$, (\ref{linearGr}) and (\ref{acotacionbb}), it follows
that
\begin{equation*}
\nonumber
C_0(z) \vert \zeta \vert - D_1(z) \leq h(z, \zeta) \leq M  z^{ m}\vert \zeta
\vert ,
\end{equation*}
for any $(z,\zeta)\in [0,\infty[\times \R^N$, where
$D_1(z) = D_0(z)+f(z,0)$.
We also assume also that the recession functions $f^0$, $h^0$ exist. Other technical assumptions on $f,h$ are required and we refer to
\cite{ACMMRelat,CEU2} for details. When we say that assumption (H) holds, we refer to the complete set of assumptions.

For the generalized relativistic heat
equation (\ref{modelo1}) the function
\begin{equation}\label{funct:frhe}
f(z,\zeta) = \frac{c^2}{\nu}  z^m \sqrt{z^2 +
\frac{\nu^2}{c^2} \vert \zeta\vert^2}
\end{equation}
satisfies all the assumptions that allow to
work in the context of entropy solutions (see
\cite{ACMEllipticFLDE,ACMMRelat}). In this case
\begin{equation*}\label{funct:arhe}\a(z,\zeta) = \nu
\frac{ z^m \zeta}{\sqrt{z^2 + \frac{\nu^2}{c^2} \vert
\zeta\vert^2}} \quad \hbox{\rm and } \quad h(z,\zeta) = \a(z,\zeta)\cdot \zeta = \nu \frac{z^m
\vert \zeta\vert^2}{\sqrt{z^2 + \frac{\nu^2}{c^2} \vert \zeta\vert^2}}.
\end{equation*}

Due to the linear growth condition on the Lagrangian, the natural
energy space to study the solutions of (\ref{DirichletproblemP}) is the
space of functions of bounded variation, or $BV$ functions. In Section \ref{sect:bv} we recall some basic basic facts about them.

The notion of entropy solutions is based on a set of Kruzkov's type inequalities and it requires to define a
functional calculus for functions whose truncations are in BV. We briefly review in Section \ref{sect:functionalcalculus}
this functional calculus which is based on the works
\cite{Dalmaso,DCFV}, which prove lower semicontinuity results for functionals on $BV$.
After this, in Section \ref{sect:defESpp} we state without proof an existence and uniqueness result for entropy solutions of (\ref{DirichletproblemP}).
The proof can be obtained by a suitable adaptation of the techniques in \cite{limitedFKPP}.
Since the traveling wave solutions we construct are functions in $L^\infty(\R^N)^+$, we give a uniqueness
result for solutions in that space (see Section \ref{SectSSS}).
A similar result was proved in \cite{limitedFKPP} for the case $m=1$.

This Section gives the necessary background for the characterization of entropy conditions given in Section \ref{sect:analysisEC}.

\subsection{Functions of bounded variation and some generalizations}\label{sect:bv}

Denote by ${\mathcal L}^N$ and
${\mathcal H}^{N-1}$ the $N$-dimensional Lebesgue measure and the
$(N-1)$-dimensional Hausdorff measure in $\R^N$, respectively.
Given an open set $\Omega$ in $\R^N$  we denote by ${\mathcal
D}(\Omega)$  the space of infinitely differentiable functions with
compact support in $\Omega$. The space of continuous functions
with compact support in $\R^N$ will be denoted by $C_c(\R^N)$.

Recall that if $\Omega$
is an open subset of $\R^N$, a function $u \in L^1(\Omega)$ whose
gradient $Du$ in the sense of distributions is a vector valued
Radon measure with finite total variation in $\Omega$ is called a
{\it function of bounded variation}. The class of such functions
will be denoted by $BV(\Omega)$.  For $u \in BV(\Omega)$, the
vector measure $Du$ decomposes into its absolutely continuous and
singular parts $Du = D^a u + D^s u$. Then $D^a u = \nabla u \
\L^N$, where $\nabla u$ is the Radon--Nikodym derivative of the
measure $Du$ with respect to the Lebesgue measure $\L^N$. We also
split $D^su$ in two parts: the {\it jump} part $D^j u$ and the
{\it Cantor} part $D^c u$. It is well known (see for instance
\cite{Ambrosio}) that $$D^j u = (u^+ - u^-) \nu_u \H^{N-1} \res
J_u,$$ where $u^+(x),u^-(x)$ denote the upper and lower approximate limits of $u$ at $x$,
$J_u$ denotes the set of approximate jump points of
$u$ (i.e. points $x\in \Omega$ for which $u^+(x)\neq u^-(x)$), and $\nu_u(x) = \frac{Du}{\vert D u \vert}(x)$,
 being $\frac{Du}{\vert D u \vert}$ the Radon--Nikodym derivative of
$Du$ with respect to its total variation $\vert D u \vert$. For
further information concerning functions of bounded variation we
refer to \cite{Ambrosio}.

We need to consider the following truncation functions. For $a <
b$, let $T_{a,b}(r) := \max(\min(b,r),a)$.
We denote
$$\mathcal T_r:= \{ T_{a,b} \ : \ 0 < a < b \}.  \ \ \
$$

Given any function $w$ and $a,b\in\R$ we shall use the notation
$\{w\geq a\} = \{x\in \R^N: w(x)\geq a\}$, $\{a \leq w\leq b\} =
\{x\in \R^N: a \leq w(x)\leq b\}$, and similarly for the sets $\{w
> a\}$, $\{w \leq a\}$, $\{w < a\}$, etc.

We need to consider the following function space
$$TBV_{\rm r}^+(\R^N):= \left\{ w \in L^1(\R^N)^+  \ :  \ \ T_{a,b}(w) - a \in BV(\R^N), \
\ \forall \ T_{a,b} \in \mathcal T_r \right\}.$$
Notice that $TBV_{\rm r}^+(\R^N)$ is
closely related to the space $GBV(\R^N)$  of generalized functions
of bounded variation introduced by E. Di Giorgi and L. Ambrosio (see \cite{Ambrosio})
Using the chain rule for
BV-functions (see for instance \cite{Ambrosio}), one can give a sense to
$\nabla u$ for a function $u \in
TBV^+(\R^N)$ as the unique function $v$ which satisfies
\begin{equation*}\label{E1WRN}
\nabla T_{a,b}(u) = v \1_{\{a < u  < b\}} \ \ \ \ \ {\mathcal
L}^N-{\rm a.e.}, \ \ \forall \ T_{a,b} \in \mathcal{T}_r.
\end{equation*}
We refer to Lemma 2.1 of \cite{Benilanetal}  or \cite{Ambrosio} for details.

\subsection{Functionals defined on BV}\label{sect:functionalcalculus}

In order to define the notion of entropy solutions of (\ref{DirichletproblemP})
and give a characterization of them,
we need a functional calculus defined on functions whose truncations are in $BV$.

Let $\Omega$ be an open subset of $\R^N$. Let $g: \Omega \times \R
\times \R^N \rightarrow [0, \infty[$ be a Borel function such that
\begin{equation*}\label{LGRWTH}
C(x) \vert \zeta \vert - D(x) \leq g(x, z, \zeta)  \leq M'(x) + M
\vert \zeta \vert
\end{equation*}
for any $(x, z, \zeta) \in \Omega \times \R \times \R^N$, $\vert
z\vert \leq R$, and any $R>0$, where $M$ is a positive constant and  $C,D,M' \geq
0$ are bounded Borel functions which may depend on $R$. Assume
that $C,D,M' \in L^1(\Omega)$.

Following Dal Maso \cite{Dalmaso} we consider the
functional:
\begin{eqnarray*}\label{RelEnerg}
{\mathcal R}_g(u)&:=& \displaystyle\int_{\Omega} g(x,u(x), \nabla u(x))
\, dx + \int_{\Omega} g^0 \left(x, \tilde{u}(x),\frac{Du}{\vert D
u \vert}(x) \right) \,  \vert D^c u \vert
\nonumber \\
&&+ \displaystyle\int_{J_u} \left(\int_{u_-(x)}^{u_+(x)}
g^0(x, s, \nu_u(x)) \, ds \right)\, d \H^{N-1}(x),
\end{eqnarray*}
for $u \in BV(\Omega) \cap L^\infty(\Omega)$, being $\tilde{u}$ is the approximated limit of $u$ \cite{Ambrosio}. The recession function $g^0$ of $g$ is defined by
\begin{equation*}\label{Asimptfunct}
 g^0(x, z, \zeta) = \lim_{t \to 0^+} tg \left(x, z, \frac{\zeta}{t}
 \right).
\end{equation*}
It is convex and homogeneous of degree $1$ in $\zeta$.

In case that $\Omega$ is a bounded set, and under standard
continuity and coercivity assumptions,  Dal Maso proved in
\cite{Dalmaso} that ${\mathcal R}_g(u)$ is $L^1$-lower semi-continuous
for $u \in BV(\Omega)$. More recently, De Cicco, Fusco, and Verde
\cite{DCFV} have obtained a very general result about the
$L^1$-lower semi-continuity of ${\mathcal R}_g$ in $BV(\R^N)$.

Assume that $g:\R\times \R^N \to [0, \infty[$ is a Borel function
such that
\begin{equation}\label{LGRWTHnox}
C \vert \zeta \vert - D \leq g(z, \zeta)  \leq M(1+ \vert \zeta \vert)
\qquad \forall (z,\zeta)\in \R^N, \, \vert z \vert \leq R,
\end{equation}
for any $R > 0$ and for some constants  $C,D,M \geq 0$ which may depend on $R$.
Observe that both functions
$f,h$ defined in (\ref{funct:frhe}), (\ref{defh}) satisfy (\ref{LGRWTHnox}).

Assume that
$$\1_{\{u\leq a\}} \left(g(u(x), 0) - g(a, 0)\right), \1_{\{u \geq b\}} \left(g(u(x),
0) - g(b, 0) \right) \in L^1(\R^N),$$
for any $u\in L^1(\R^N)^+$.
Let $u \in TBV_{\rm r}^+(\R^N)  \cap L^\infty(\R^N)$  and $T = T_{a,b}\in {\mathcal T}_r $.
For each $\phi\in
C_c(\R^N)$, $\phi \geq 0$, we define the Radon measure  $g(u, DT(u))$ by
\begin{eqnarray}\label{FUTab}
\langle g(u, DT(u)), \phi \rangle &: =& {\mathcal R}_{\phi g}(T_{a,b}(u))+
\displaystyle\int_{\{u \leq a\}} \phi(x)
\left( g(u(x), 0) - g(a, 0)\right) \, dx  \nonumber\\
&& \displaystyle  + \int_{\{u \geq b\}} \phi(x)
\left(g(u(x), 0) - g(b, 0) \right) \, dx.
\end{eqnarray}
If $\phi\in C_c(\R^N)$, we write $\phi = \phi^+ -
\phi^-$ with $\phi^+= \max(\phi,0)$, $\phi^- = - \min(\phi,0)$,
and we define $\langle g(u, DT(u)), \phi \rangle : =
\langle g(u, DT(u)), \phi^+ \rangle- \langle g(u, DT(u)), \phi^- \rangle$.

Recall that, if $g(z,\zeta)$ is continuous in $(z,\zeta)$, convex
in $\zeta$ for any $z\in \R$, and $\phi \in C^1(\R^N)^+$ has compact
support, then  $\langle g(u, DT(u)), \phi \rangle$ is lower
semi-continuous in $TBV^+(\R^N)$ with respect to
$L^1(\R^N)$-convergence \cite{DCFV}. This property is used to prove existence of
solutions of (\ref{DirichletproblemP}).

We can now define the required functional calculus. We follow \cite{CEU2} and note that it
represents an extension of the functional calculus in
\cite{ACMEllipticFLDE,ACMMRelat} that uses a more restrictive class of test functions.

Let us denote by ${\mathcal P}$ the set of Lipschitz continuous functions $p : [0, +\infty[ \rightarrow \R$
satisfying $p^{\prime}(s) = 0$ for $s$ large enough. We write
${\mathcal P}^+:= \{ p \in {\mathcal P} \ : \ p \geq 0 \}$.

Let $S \in C([0,\infty[)$ and $p \in {\mathcal P} \cap C^1([0,\infty[)$.
We denote
$$
f_{S:p}(z,\zeta) = S(z)p'(z) f(z,\zeta), \qquad h_{S:p}(z,\zeta) = S(z)p'(z)h(z,\zeta).
$$
If $Sp'\geq 0$, then the function $f_{S:p}(z,\zeta)$ satisfies the assumptions implying the lower
semicontinuity of the associated energy functional \cite{DCFV}.

Assume that $p(r)=p(T_{a,b}(r))$, $0 < a < b$. We assume that
$u \in TBV_{\rm r}^+(\R^N)$ and
$$
\1_{\{u\leq a\}} S(u)\left(f(u(x), 0) - f(a, 0)\right), \1_{\{[u \geq b\}} S(u)\left(f(u(x),
0) - f(b,0) \right) \in L^1(\R^N).             
$$
Since $h(z, 0) = 0$, the last assumption clearly holds for $h$.

Finally, we define
$f_{S:p}(u,DT_{a,b}(u)),$  $h_{S:p}(u,DT_{a,b}(u))$ as the Radon measures given by
(\ref{FUTab}) with $g(z,\zeta) = f_{S:p}(z,\zeta)$ and $g(z,\zeta) = h_{S:p}(z,\zeta)$, respectively.

\subsection{Existence and uniqueness of entropy solutions}\label{sect:defESpp}

\subsubsection{The class of  test functions}

Let us introduce the class of test functions required to define entropy sub- and super-solutions.
If $u\in TBV_{\rm r}^+(\R^N)$,
we define $\mathcal{TSUB}$ (resp. $\mathcal{TSUPER}$ ) as the class of functions
$S,T \in \mathcal{P}$ such that
$$S \geq 0 , S'\geq 0 \quad \hbox{\rm and}\quad T\geq 0,T'\geq 0,$$
$$(\hbox{\rm resp.}\,  S \leq 0, S'\geq 0 \quad \hbox{\rm and}\quad T\geq 0,T'\leq 0)$$
and $p(r) = \tilde{p}(T_{a,b}(r))$ for some $0 < a < b$, where
$\tilde{p}$ is differentiable in a neighborhood of $[a,b]$ and $p$ represents either $S$ or $T$.

Although the proof of uniqueness and the development of the theory requires only the use of test functions $S,T\in\mathcal{T}^+$
and this was the family used in \cite{ACMMRelat},
the analysis of the entropy conditions is facilitated by the use of more general test functions
in $\mathcal{TSUB}$ and $\mathcal{TSUPER}$.

\subsubsection{Entropy solutions in $L^1\cap L^\infty$.}
\label{sect:B}
Let $L^1_{w}(0,T,BV(\R^N))$  be the space of weakly$^*$
measurable functions $w:[0,T] \to BV(\R^N)$ (i.e., $t \in [0,T]
\to \langle w(t),\phi \rangle$ is measurable for every $\phi$ in the predual
of $BV(\R^N)$) such that $\int_0^T \Vert w(t)\Vert_{BV} \, dt< \infty$.
Observe that, since $BV(\R^N)$ has a separable predual (see
\cite{Ambrosio}), it follows easily that the map $t \in [0,T]\to
\Vert w(t) \Vert_{BV}$ is measurable. By  $L^1_{loc, w}(0, T,
BV(\R^N))$ we denote the space of weakly$^*$ measurable functions
$w:[0,T] \to BV(\R^N)$ such that the map $t \in [0,T]\to \Vert
w(t) \Vert_{BV}$ is in $L^1_{loc}(]0, T[)$.

\begin{definition} \label{def:espb}
Assume that $u_0 \in (L^1(\R^N)\cap L^\infty(\R^N))^+$.
A measurable function $u: ]0,T[\times \R^N \rightarrow \R$ is an
{\it entropy sub-solution} (resp. {\it super-solution}) of (\ref{DirichletproblemP}) in $Q_T =
]0,T[\times \R^N$ if $u \in C([0, T]; L^1(\R^N))$,
$T_{a,b}(u(\cdot)) - a \in L^1_{loc, w}(0, T, BV(\R^N))$ for all $0 <
a < b$, and

\begin{itemize}
\item[(i)]  $u(0) \leq u_0$ (resp.  $u(0) \geq u_0$), and
\item[(ii)] \ the following inequality is satisfied
\begin{eqnarray}\label{pei}
&& \hspace{-0.6cm}\displaystyle \int_0^T\int_{\R^N} \phi
h_{S:T}(u,DT_{a,b}(u)) \, dt + \int_0^T\int_{\R^N} \phi h_{T:S}(u,DS_{c,d}(u)) \, dt \nonumber
 \\ &&\hspace{-0.2cm}\leq  \displaystyle\int_0^T\int_{\R^N} \Big\{ J_{TS}(u(t)) \phi^{\prime}(t) - \a(u(t), \nabla u(t)) \cdot \nabla \phi \
T(u(t)) S(u(t))\Big\} dxdt  \nonumber
 \\
&& \ +\displaystyle\int_0^T\int_{\R^N}\phi(t)T(u(t))S(u(t))F(u(t))
\, dxdt,
\end{eqnarray}
 for truncation functions $(S , \, T) \in \mathcal{TSUB}$ (resp. $(S , \, T) \in \mathcal{TSUPER}$)
 with $T=\tilde T\circ T_{a,b}$, $S=\tilde S\circ S_{c,d}$, $0 < a < b$, $0 < c < d$,  and any  smooth function $\phi$ of
 compact support, in particular  those  of the form $\phi(t,x) =
 \phi_1(t)\rho(x)$, $\phi_1\in {\mathcal D}(]0,T[)$, $\rho \in
 {\mathcal D}(\R^N)$.
\end{itemize}
We say that $u: ]0,T[\times \R^N \rightarrow \R$ is an
{\it entropy  solution}   of (\ref{DirichletproblemP}) if it is an entropy sub- and super-solution.
\end{definition}

Notice that if $u$ is an entropy sub-solution (resp. super-solution), then
$u_t \leq  {\rm div} \, \a(u(t), \nabla u(t)) + F(u(t))$ (resp. $\geq$) in $\mathcal{D}^\prime(Q_T)$.
We notice also that $u$ is an entropy solution if
$u_t =  {\rm div} \, \a(u(t), \nabla u(t)) + F(u(t))$  in $\mathcal{D}^\prime(Q_T)$, $u(0)=u_0$
and the inequalities  (\ref{pei}) hold for truncations $(S , \, T) \in \mathcal{TSUB}$ and any test functions
as in $(ii)$ \cite{CEU2}.

We have the following existence and uniqueness result, which is an extension of those in  \cite{limitedFKPP}.

\begin{theorem}\label{EUTEparabolic}
Let the set of assumptions
(H) be satisfied and let $F$ be Lipschitz continuous with $F(0) = 0$.
Then, for any initial datum $0 \leq u_0 \in L^{\infty}(\R^N) \cap
L^{1}(\R^N)$ there exists a unique entropy solution $u$ of
(\ref{DirichletproblemP}) in $Q_T$ for every $T
> 0$ such that $u(0) = u_0$, satisfying $u
\in C([0, T]; L^1(\R^N))$  and $F(u(t)) \in L^1(\R^N)$ for almost all
$0 \leq t \leq T$.  Moreover, if $u(t)$, $\overline{u}(t)$  are
entropy solutions corresponding to initial data $u_0$,
$\overline{u}_0 \in \left(L^{\infty}(\R^N) \cap
L^{1}(\R^N)\right)^+$, respectively, then
\begin{equation*}
\label{CPentropys}\Vert u(t) - \overline{u}(t) \Vert_1 \leq
e^{t\Vert F \Vert_{Lip}} \, \Vert u_0 - \overline{u}_0 \Vert_1 \ \
\ \ \ \ {\rm for \ all} \ \ t \geq 0.
\end{equation*}
\end{theorem}

\subsubsection{Entropy solutions in $L^\infty$}\label{SectSSS}

In order to cover the case of bounded traveling waves, we extend the  notion of entropy solutions to
functions in $L^{\infty}(\R^N)^+$.
We follow the presentation in \cite{limitedFKPP}.

\begin{definition}\label{SES} Given  $0 \leq u_0 \in L^{\infty}(\R^N)$, we
say that a measurable function $u: ]0,T[\times \R^N \rightarrow
\R$  is an {\it entropy sub-solution} (respectively, {\it entropy
super-solution}) of the Cauchy problem (\ref{DirichletproblemP}) in
$Q_T = ]0,T[\times \R^N$ if  $u
\in C([0, T]; L_{loc}^1(\R^N))$, $u(0)\leq u_0$ (resp. $u(0) \geq
u_0$), $F(u(t)) \in
L_{loc}^1(\R^N)$  for almost {every} $0 \leq t \leq T$,
$T_{a,b}(u(\cdot)) - a \in L^1_{loc, w}(0, T, BV_{\rm loc}(\R^N))$
for all $0 < a < b$, $\a(u( \cdot), \nabla u(\cdot)) \in
L^{\infty}(Q_T)$, and the  inequalities (\ref{pei}) are satisfied
 for truncations $(S , \, T) \in \mathcal{TSUB}$ (resp. $(S , \, T) \in \mathcal{TSUPER}$)
 with $T=\tilde T\circ T_{a,b}$, $S=\tilde S\circ S_{c,d}$, $0 < a < b$, $0 < c < d$, and any  smooth function $\phi$ of
 compact support, in particular  those  of the form $\phi(t,x) =
 \phi_1(t)\rho(x)$, $\phi_1\in {\mathcal D}(]0,T[)$, $\rho \in
 {\mathcal D}(\R^N)$.

We say that $u: ]0,T[\times \R^N \rightarrow
\R$  is an solution of (\ref{DirichletproblemP}) if $u$ is an entropy sub-solution and
super-solution.
\end{definition}

\begin{definition}\label{def:nullflux}
Let $u$ be a sub- or a  super-solution  of
(\ref{DirichletproblemP}) in $Q_T$. We say that $u$ has a null
flux at infinity if
$$
\lim_{R \to + \infty} \int_0^T  \int_{ \R^N} \vert \a(u(t),\nabla
u(t)) \vert \, \vert \nabla \psi_R (x) \vert \, dx dt = 0
$$
for all  $\psi_R \in {\mathcal D}(\R^N)$ such that $0 \leq \psi_R
\leq 1$, $\psi_R \equiv 1$ on $B_R$, $\hbox{supp}(\psi_R) \subset
B_{R+2}$ and $\Vert \nabla \psi_R \Vert_{\infty} \leq 1$.
\end{definition}

We have uniqueness of entropy solutions for initial data in $L^{\infty}(\R^N)$ when
they have null flux at infinity.

\begin{theorem}\label{UniqSup} Let the set of assumptions
(H) be satisfied and let $F$ be Lipschitz continuous with $F(0) = 0$.
\begin{itemize}
\item[(i)] Let  $u(t)$,
$\overline{u}(t)$ be two entropy solutions of (\ref{DirichletproblemP})
with initial data $u_0, \overline{u}_0  \in L^{\infty}(\R^N)^+$, respectively.
Assume that  $u(t)$ and $\overline{u}(t)$ have null flux at infinity.
Then
\begin{equation*}
\label{CPeUNIQ} \Vert u(t) - \overline{u}(t) \Vert_1 \leq e^{t\Vert F \Vert_{Lip}} \,
\Vert u_0 - \overline{u}_0 \Vert_1, \ \ \ \ \ {\rm for \ all} \ \ t \geq 0.
\end{equation*}
\item[(ii)]
Assume that $u_0\in (L^1(\R^N)\cap L^\infty(\R^N))^+$, $\overline{u}_0  \in L^{\infty}(\R^N)^+$.
Let $u(t)$ be the entropy solution of (\ref{DirichletproblemP})
with initial datum $u_0$. Let $\overline{u}(t)$ be an entropy super-solution of (\ref{DirichletproblemP})
with initial datum $\overline{u}_0  \in L^{\infty}(\R^N)^+$ having a null flux at infinity.
Assume in addition that $\overline u(t) \in BV_{\rm loc}(\R^N)$ for almost every $0 < t < T$. Then
\begin{equation*}
\label{CPeUNIQcp} \Vert (u(t) - \overline{u}(t))^+ \Vert_1 \leq e^{t\Vert F \Vert_{Lip}} \,
\Vert (u_0 - \overline{u}_0)^+ \Vert_1, \ \ \ \ \ {\rm for \ all} \ \ t \geq 0.
\end{equation*}
\end{itemize}
\end{theorem}

\bibliographystyle{amsplain}

\end{document}